\documentclass[12pt]{amsart}
\usepackage{latexsym,amsxtra,amscd,ifthen,amsmath,color, multicol,hyperref,mathdots,mathrsfs}
\usepackage{amsfonts}
\usepackage{verbatim}
\usepackage{amsmath}
\usepackage{amsthm}
\usepackage{amssymb}
\usepackage[all,cmtip]{xy}
\usepackage[normalem]{ulem}
\usepackage{enumerate}
\usepackage[latin1]{inputenc}
\usepackage{amsmath}
\usepackage{amsfonts}
\usepackage{amssymb}
\usepackage{graphicx,mathtools}
\usepackage{multirow}
\usepackage{latexsym}
\usepackage{amssymb}
\usepackage{placeins} %set a barrier which may not be crossed by floats, to fix positions of figures and tables
\allowdisplaybreaks

\usepackage{graphicx}
\usepackage{tikz}
\usetikzlibrary{mindmap,shapes,snakes,matrix}

\theoremstyle{plain}
\newtheorem{thm}{Theorem}[section]
\newtheorem{lem}[thm]{Lemma}
\newtheorem{prop}[thm]{Proposition}
\newtheorem{cor}[thm]{Corollary}

\theoremstyle{definition}

\newtheorem{deff}[thm]{Definition}

\newtheorem{example}[thm]{Example}
\newtheorem{remark}[thm]{Remark}

\theoremstyle{remark}

\def\k{\ensuremath{\bold{k}}}

\newcommand*{\e}{\ensuremath{\varepsilon}}

\newcommand*{\field}{\ensuremath{\bold{k}}}
\newcommand*{\Hom}{\ensuremath{\text{\upshape Hom}}}

\newcommand*{\Ext}{\ensuremath{\text{\upshape Ext}}}
\newcommand*{\Ker}{\ensuremath{\text{\upshape Ker}}}
\newcommand*{\Img}{\ensuremath{\text{\upshape Im}}}

\newcommand*{\ad}{\ensuremath{\text{\upshape ad}}}
\newcommand*{\Prim}{\ensuremath{\text{\upshape P}}}
\newcommand*{\End}{\ensuremath{\text{\upshape End}}}

\newcommand*{\Aut}{\ensuremath{\text{\upshape Aut}}}
\newcommand*{\Diag}{\ensuremath{\text{\upshape diag}}}
\newcommand*{\Der}{\ensuremath{\text{\upshape Der}}}

\newcommand*{\GL}{\ensuremath{\text{\upshape GL}}}

\newcommand*{\HL}{\ensuremath{\text{\upshape H}}}

\newcommand*{\T}{\ensuremath{\text{\upshape T}}}
\newcommand*{\C}{\ensuremath{\text{\upshape C}}}

\def\dim{\operatorname{dim}}

\newcommand*{\Bock}{\boldsymbol{\omega}}

\DeclarePairedDelimiterX\set[1]\lbrace\rbrace{\def\given{\;\delimsize\vert\;}#1}

\begin{document}
\numberwithin{equation}{thm}
\thispagestyle{empty}

\title{Primitive Deformations of Quantum $p$-groups}

\author{Van C. Nguyen\,*}
\author{Linhong Wang}
\author{Xingting Wang}

\address{Department of Mathematics\\
Northeastern University\\
Boston, MA 02115}
\email{v.nguyen@neu.edu}

\address{Department of Mathematics\\
University of Pittsburgh \\
Pittsburgh, PA 15260}
\email{lhwang@pitt.edu}

\address{Department of Mathematics\\
Temple University, Philadelphia\\
Philadelphia, PA,  19122-6094}
\email{xingting@temple.edu}

\thanks{* corresponding author.}



\begin{abstract}
For finite-dimensional Hopf algebras, their classification in characteristic $0$ (e.g. over $\mathbb{C}$) has been investigated for decades with many fruitful results, but their structures in positive characteristic have remained elusive.  In this paper, working over an algebraically closed field $\field$ of prime characteristic $p$, we introduce the concept, called Primitive Deformation, to provide a structured technique to classify certain finite-dimensional connected Hopf algebras which are almost primitively generated; that is, these connected Hopf algebras are $p^{n+1}$-dimensional, whose primitive spaces are abelian restricted Lie algebras of dimension $n$. We illustrate this technique for the case $n=2$. Together with our results in \cite{NWW1}, we provide a complete classification of $p^3$-dimensional connected Hopf algebras over $\field$ of characteristic $p>2$.
\end{abstract}

\maketitle

%%%%%%%%%%%%%%%%%%%%%%%%%%%%%%%%

\section*{Introduction}

The classification of finite-dimensional Hopf algebras over a field of characteristic $0$ is well investigated by many researchers, among them are: Andruskiewitsch, Etingof, Gelaki, Masuoka, Montgomery, Natale, Ng, Radford, Schneider, Zhu, etc., see \cite{An, Bea,BG} for surveys of classification results in characteristic $0$. While this work is stimulating in its own right, many Hopf algebras of interest are, however, defined over a field $\k$ of positive characteristic, where the classification is much less known, e.g. \cite{CLW, EGelaki, Masu}. Our aim is to explore the classification over such a field $\k$. It is important to emphasize that this is a very difficult problem. We face some formidable challenges as many of the tools used over $\mathbb{C}$ do not apply in positive characteristic. The structures of such Hopf algebras essentially vary depending on the characteristic of the field and its relation with the dimensions of the Hopf algebras. Therefore, part of the challenge is to develop new applicable classification techniques in positive characteristic. 

Hopf algebras in characteristic $0$ are completely classified for certain low dimensions and many of them turn out to be trivial, that is, they are either group algebras or the dual of group algebras. On the other hand, the Hopf algebra structures in positive characteristic are much more complicated. In \cite{EGelaki}, Etingof and Gelaki studied the structures of finite-dimensional semisimple and cosemisimple Hopf algebras in positive characteristic, using their celebrated lifting theorem that allows them to lift these algebras to some semisimple Hopf algebras in characteristic $0$. In \cite{S}, Scherotzke classified finite-dimensional pointed rank one Hopf algebras in positive characteristic which are generated by group-like and skew-primitive elements. Gongxiang and Yu \cite{GY} studied the structures of monomial Hopf algebras over a field of positive characteristic and provided a necessary and sufficient condition for the monomial coalgebra to admit Hopf structures. Many Hopf algebras in positive characteristic also come from Nichols algebras which are of much interest, see \cite{CLW, HW, OZ}. In each of these results, we see that different techniques must be employed and new examples of Hopf algebras arise. There are more Hopf structure possibilities in positive characteristic than occur in characteristic $0$. This calls for a more general approach to tackle certain types of Hopf algebras in positive characteristic. We present here our attempt for this difficult task. 

%\vspace{-0.5em}
\FloatBarrier
$$
\begin{tikzpicture}
%\tikzstyle{iellipse}=[draw=black,shape=ellipse,very thick,fill=green!18!white];  %The fill color "green!20!white" means 20% green and 80% white mixed together.
\tikzstyle{ibox}=[draw=black,shape=rectangle,very thick,fill=red!18];
\node[ibox, align=center, below] (n1) at (-2,0)  {\textbf{Goal:} Classify finite-dimensional connected Hopf algebras in characteristic $p>0$.}; %Multiline text can be included inside a node. A new line is indicated by double backslash "\\", but additionally you have to specify the alignment using the node option "align=". 
\end{tikzpicture}
$$
These algebras are interesting in their own right; they all have dimensions of powers of $p$ \cite[Proposition 1.1 (1)]{Masu}, and they are closely related to Lie theory, algebraic group schemes, and group theory, cf.~Section~\ref{sec:prelim}.

We fix the notation and recall preliminary concepts in Section~\ref{sec:prelim}. In Section~\ref{sec:PD theory}, we introduce a technique called Primitive Deformation, in which we think of finite-dimensional connected Hopf algebras as finite quantum $p$-groups (Definition~\ref{quantum}) and we apply our technique to the special family (cf.~\cite[Definition~0.2]{XWang}):

\begin{center}
$\mathscr X(n) := \displaystyle \{$all $p^{n+1}$-dimensional quantum $p$-groups whose primitive space \\ \quad \quad is an abelian restricted Lie algebra of dimension $n \}.$
\end{center}

In \cite{XWang}, quantum $p$-groups in $\mathscr X(n)$ arise as Hopf deformations of certain restricted universal enveloping algebras, whose underlying restricted Lie algebras are obtained by semiproducts of restricted Lie algebras presented by abelian types $\T=(\mathfrak g, \mathfrak h, \rho)$, cf.~Definition~\ref{D:Type}, with restriction on the resulting primitive space. Moreover, using primitive cohomology, isomorphism classes of quantum $p$-groups can be described as a cohomological-type set quotient by the automorphism group of each abelian type. %by taking deformations (on primitive elements) of certain abelian Hopf extensions of the restricted universal enveloping algebras of restricted Lie algebras. 
Our strategy is to consider all possible abelian types $\T$. We then give conditions under which these types $\T$ can be associated to primitive deformations (Proposition~\ref{P:RCdata}). By Theorem~\ref{TH:HopfO}, equivalence classes of these primitive deformations are in 1-1 correspondence with the isomorphism classes of quantum $p$-groups in $\mathscr X(n)$. We describe these classes explicitly as some fixed cohomological-type set in Section~\ref{sec:iso classes of PD} and further realize these classes as certain points in affine spaces. This helps us to break our classification problem into permissible types and nonpermissible types, see Section~\ref{sec:realization}. We illustrate this concept for $\mathscr X(2)$, with $p>2$, in Section~\ref{sec:X(2)} by directly computing and explicitly describing all $p^3$-dimensional quantum $p$-groups in $\mathscr X(2)$. We observe that the classification in $\mathscr X(2)$ contains \emph{infinite} parametric families and leads us to many new examples of non-commutative, non-cocommutative Hopf algebras in positive characteristic. It further helps us to complete the classification of all $p^3$-dimensional connected Hopf algebras, whose structures are provided in Appendix~\ref{Appen:complete classification}. We summarize the classification results here:  

\newpage
\FloatBarrier
\begin{table}[!htp]%[htp] % table is a floating environment. TeX has special rules for placing floats so that they cause minimal disruption of the surrounding text. You can give LaTeX hints of how to place your floats. For example, to make LaTeX try to place it where you inserted it, specify the option h (for ?here?). Note that this won't always work because LaTeX still tries to make the text look nice. Your hint is only that: a hint, not a definite command. You can force LateX to obey your command to place a floating environment HERE by loading the package "float" and specifying the option H (capital letter!). But notice that this can lead to an ugly layout. 
\begin{center}%\small\setlength\tabcolsep{2pt}
\caption{Classification of connected $p^3$-dim Hopf algebras}  
\label{tab:summary} \vspace{-0.2cm}
\begin{tabular}{| p{2.1cm} | p{2.5cm} | p{3.2cm}|} 
\hline         
$\Prim(H)$ & $u(\Prim(H)) \subseteq H$ & \bf{Results}    \\
\hline \hline
1-dim & $p$-dim & \cite[Theorem 1.1]{NWW1}, Table \ref{tab:typeA} \\ \hline
2-dim, non-abelian & $p^2$-dim, non-commutative &  \cite[Theorem 1.3]{NWW1}, Table \ref{tab:typeB}  \\ \hline
2-dim, abelian & $p^2$-dim, $p>2$, commutative &  Section~\ref{sec:X(2)}, Table \ref{tab:typeT} \\ \hline
3-dim & $p^3$-dim &  \cite[Theorem 1.4]{NWW1}, Table \ref{tab:typeC}  \\ \hline
\end{tabular}
\end{center}
\end{table}

Thus, together with results in \cite{wang2012connected}, the classifications in characteristic $p>2$ for connected Hopf algebras of dimensions $p$, $p^2$, and $p^3$ are completed. This milestone brings us closer to understanding the structures of Hopf algebras in positive characteristic. To our best knowledge, the current status of Hopf classifications in characteristic $p>0$ can be described in the following diagram:

%Likewise, our results here will not only contribute to the classification area but also provide interesting Hopf algebra examples to related fields where they become applicable.

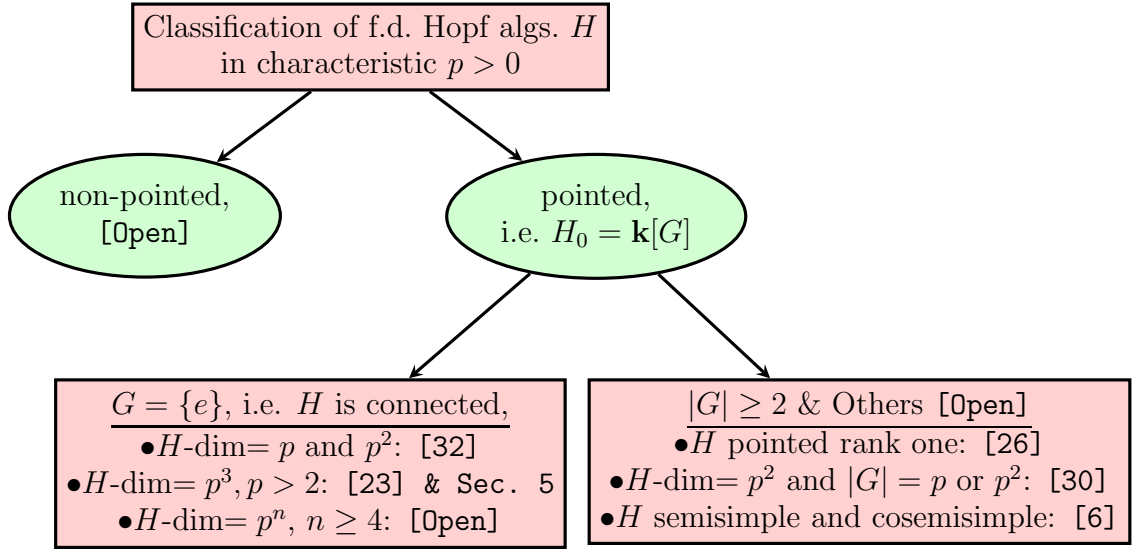
\FloatBarrier
$$
\begin{tikzpicture}
\tikzstyle{iellipse}=[draw=black,shape=ellipse,very thick,fill=green!18!white];  %The fill color "green!20!white" means 20% green and 80% white mixed together.
\tikzstyle{ibox}=[draw=black,shape=rectangle,very thick,fill=red!18];
\node[ibox, align=center, below] (n1) at (2,0)  {Classification of f.d.~Hopf algs.~$H$ \\ in characteristic $p>0$}; %Multiline text can be included inside a node. A new line is indicated by double backslash "\\", but additionally you have to specify the alignment using the node option "align=". 
\node[iellipse, align=center, below] (n2) at (-1,-2)  {non-pointed, \\ \texttt{[Open]}};
\node[iellipse, align=center, below] (n3) at (5,-2)  {pointed, \\ i.e.~$H_0=\field[G]$};
\node[ibox, align=center, below] (n4) at (1.2,-5)  {\underline{$G=\{e\}$, i.e. $H$ is connected,} \\ $\bullet H$-dim$=p$ and $p^2$: \texttt{\cite{wang2012connected}} \\ $\bullet H$-dim$=p^3, p>2$: \texttt{\cite{NWW1} \& Sec.~\ref{sec:X(2)}} \\ $\bullet H$-dim$=p^n, \, n \geq 4$: \texttt{[Open]}};
\node[ibox, align=center, below] (n5) at (8.5,-5)  {\underline{$|G| \geq 2$ \& Others \texttt{[Open]}} \\ $\bullet H$ pointed rank one: \texttt{\cite{S}} \\ $\bullet H$-dim$=p^2$ and $|G|=p$ or $p^2$: \texttt{\cite{WangWang}} \\ $\bullet H$ semisimple and cosemisimple: \texttt{\cite{EGelaki}} };
%\draw[->, ultra thick, >=stealth,color=blue] (n1) -- (n2); 
\draw[->, very thick, >=stealth] (n1) -- (n2);
\draw[->, very thick, >=stealth] (n1) -- (n3);
\draw[->, very thick, >=stealth] (n3) -- (n4);
\draw[->, very thick, >=stealth] (n3) -- (n5);
\end{tikzpicture}
$$

\noindent
{\bf Applications.}
Our classification results in Section~\ref{sec:X(2)} provide some new examples of connected Hopf algebras of dimension $p^3$ in positive characteristic. We remark two other important applications that this work carries:
\begin{enumerate}
 \item By \cite[Proposition~5.2.9]{MO93}, for any connected coalgebra $C$, its $\k$-dual $C^*$ is a local algebra. We can derive from this result that any finite-dimensional Hopf algebra $H$ is local if and only if its dual $H^*$ is connected. Hence, our classifications are equivalent to the classifications of \emph{local} Hopf algebras of dimension $p^3$.
 \vspace{0.2cm}
 \item By \cite[\S 8.3]{Waterhouse}, an affine algebraic group scheme $G$ is \emph{unipotent} if and only if the polynomial functions $\mathcal{O}(G)$ on $G$ is a connected {\bf commutative} Hopf algebra. Thus, our results can also be related to unipotent group schemes $G$ of order $p^3$. Moreover, we know that in characteristic $0$, unipotent group is isomorphic to an algebraic closed subgroup of some $\mathbb{U}_n$ of strict upper triangular matrices. In positive characteristic, structures of unipotent groups are more complicated. % they correspond to the commutative algebras in our classification results. 
Interested readers can refer to \cite{Milne, Waterhouse} and utilize our results to understand more about this type of algebraic groups in positive characteristic.
\end{enumerate}

%%%%%%%%%%%%%%%%%%%%%%%%%%%%%%%%

\section{Preliminary}
\label{sec:prelim}

Throughout the paper, we work over a base field $\field$, algebraically closed of prime characteristic $p$. Tensor products and linear maps are over $\field$ unless stated otherwise. For any vector space $V$, we write $V^{\otimes n}$ for its $n$-fold tensor product and 1 for the identity map on $V$. 
 
\begin{deff}\cite[Definition 5.1.5]{MO93}
Let $C$ be a coalgebra over $\field$. The \emph{coradical} $C_0$ of $C$ is the sum of all simple subcoalgebras of $C$, and $C$ is \emph{connected} if $C_0$ is one-dimensional. A Hopf algebra $H$ is said to be \emph{connected} if it is connected as a coalgebra. 
%Let $H$ be any Hopf algebra. Then $H$ is said to be \emph{connected} if it only has one non-isomorphic simple comodule, i.e., the trivial comodule $\field$.  
\end{deff}

In the literature, connected Hopf algebras are often called by different names, such as irreducible or co-connected Hopf algebras. 

\begin{example}
Let $G$ be a finite group. The vector space dual $\field ^G:=\Hom_\field(\field[G],\field)$ is equipped with a Hopf algebra structure by dualizing all the Hopf structure maps of $\field[G]$, namely, for any $g,h\in G$,
\[
\delta_g\delta_h=\begin{cases}
\delta_g, & g=h\\
0,           & g\neq h
\end{cases},
\quad
\Delta(g)=\sum_{h\in G}\delta_{gh^{-1}}\otimes \delta_{h}, \quad S(\delta_g)=\delta_{g^{-1}}, \quad \epsilon(\delta_g)=\begin{cases} 1, & g=e \\ 0, & g \neq e \end{cases},
\]
where $\delta_g$ is the dual basis of $g$ on $k[G]$, and $e$ is the identity in $G$. When $G$ is a finite p-group, $\field^G$ is a connected Hopf algebra of dimension $|G|$.
\end{example}

We recall a result by Masuoka clarifying when a finite-dimensional connected Hopf algebra $H$ is isomorphic to some $\field^G$.

\begin{thm}\cite{Masu,XWangSS}
Let $H$ be a finite-dimensional connected Hopf algebra over $\field$. Then the following are equivalent. 
\begin{itemize}
\item[(i)] $H$ is semisimple.
\item[(ii)] $H$ is commutative and semisimple.
\item[(iii)] The primitive space of $H$ is a torus.
\item[(iv)] $H\cong \field^G$, for some finite $p$-group $G$.
\end{itemize}
\end{thm}
The \emph{primitive space} of $H$ is the set $\Prim(H):=\{x\in H\, |\, \Delta(x)=x\otimes 1+1\otimes x\}$. We say that $\Prim(H)$ is a \emph{torus} %if the $p$-th power map of $H$ is surjective on $\Prim(H)$, 
if it does not contain any nonzero element $x$ with $x^p=0$, which is, by \cite[Proposition 1.3]{Masu}, equivalent to the original definition given by Strade and Farnsteiner \cite[p.~86]{StradeFarn}. Now, we propose the quantized version of a finite $p$-group. 

\begin{deff}\label{quantum} 
A finite \emph{quantum $p$-group} is a finite-dimensional connected Hopf algebra in characteristic $p$.
\end{deff}

By Masuoka's result, a semisimple quantum $p$-group over $\field$ is always given by some finite $p$-group. Moreover, any commutative quantum $p$-group over $\field$ represents a unipotent group scheme over $\field$ \cite{Waterhouse}. Another rich source of quantum $p$-groups comes from the restricted universal enveloping algebra of a restricted Lie algebra, which was first defined by Jacobson in \cite{Jac37}  as an analog of derivations of associative algebras in characteristic $p$.

\begin{deff}\cite[Definition on p.~16]{Jac41}\label{D:RLie}
A Lie algebra $\mathfrak g$ over $\field$ is \emph{restricted} if there exists a map $\mathfrak g\to \mathfrak g$, $x\mapsto x^{[p]}$ satisfying, for all $x,y\in \mathfrak g$, $\alpha\in \field$:
\begin{itemize}
\item[(i)] $(\alpha x)^{[p]}=\alpha^p x^{[p]}$,
\item[(ii)] $(x+y)^{[p]}=x^{[p]}+y^{[p]}+\sum_{i=1}^{p-1}s_i(x,y)$, where $is_i(x,y)$ is the coefficient of $\lambda^{i-1}$ in $x(\ad \, (\lambda x+y))^{p-1}$,
\item[(iii)] $[x,y^{[p]}]=x(\ad \, y)^p$. 
\end{itemize} 
\end{deff}

\begin{example}\label{E:RLA}
Let $A$ be an associative algebra over our field $\field$ of characteristic $p$. Suppose $V$ is a subspace of $A$, which is invariant under the commutator and $p$-th power map of $A$, i.e., $[V,V]\subseteq V$ and $V^p\subseteq V$. Then $V$ is a restricted Lie algebra, where the Lie bracket is given by the commutator of $A$ and the restricted map is given by the $p$-th power map of $A$. In particular, for any Hopf algebra $H$ over $\field$ and for any $x,y\in \Prim(H)$, one sees easily that 
\begin{align*}
\Delta\left(x^p\right)&=(x\otimes 1+1\otimes x)^p=x^p\otimes 1+1\otimes x^p;\ \text{and}\\ 
\Delta\left([x,y]\right)&=\left[\Delta(x),\Delta(y)\right]=[x,y]\otimes 1+1\otimes [x,y].
\end{align*}
Hence, $x^p,[x,y]\in \Prim(H)$, and the primitive space $\Prim(H)$ is a restricted Lie algebra equipped with the Lie bracket and the restricted map given by the commutator and the $p$-th power map of $H$, respectively.  
\end{example}

\begin{deff}
Let $\mathfrak g$ be a restricted Lie algebra and $U(\mathfrak \mathfrak g)$ be its universal enveloping algebra. %Denote by $B$ the ideal in $U(\mathfrak g)$ generated by all $x^p-x^{[p]}$, $x\in \mathfrak g$, and define  
The \emph{restricted universal enveloping algebra} of $\mathfrak g$ is the quotient algebra $u(\mathfrak g):=U(\mathfrak g)/\left(x^p-x^{[p]}, \, x\in \mathfrak g \right)$.
\end{deff} 

A version of the PBW theorem holds for $u(\mathfrak g)$: given a basis for $\mathfrak g$, the ordered monomials in this basis, where the exponential of each basis element is bounded by $p-1$, form a basis for $u(\mathfrak g)$. Consequently, if $\dim \mathfrak g=n$, then $\dim u(\mathfrak g)=p^n$. By setting $\Delta(x)=x\otimes 1+1\otimes x$, $S(x)=-x$ and $\epsilon(x)=0$, for every $x\in \mathfrak g\subset u(\mathfrak g)$, one sees that $u(\mathfrak g)$ becomes a finite-dimensional connected Hopf algebra over $\field$, which is a quantum $p$-group by Definition~\ref{quantum}. 

The Hochschild cohomology of pointed coalgebras was first used by Stefan and van Oystaeyen in \cite{SteVan} to classify complex pointed Hopf algebras of dimension $p^3$. Recently, the technique has been further developed by Wang-Zhang-Zhuang to study infinite-dimensional connected Hopf algebras over an algebraically closed field of characteristic zero \cite{DZZ}. Using the idea of \cite[Definition 1.2]{DZZ} and applying Lemma~\ref{L:HochUH}, we interpret the Hochschild cohomology ring of restricted universal enveloping algebras in terms of the cobar construction, which is defined as follows. 

\begin{deff}
Let $\mathfrak h$ be a restricted Lie algebra, and $u(\mathfrak h)$ be its restricted universal enveloping algebra. The \emph{cobar construction} on $u(\mathfrak h)$ is the differential graded algebra $\Omega\, u(\mathfrak h)$ such that:
\begin{itemize}
\item[(i)] As a graded algebra, $\Omega\, u(\mathfrak h)$ is the tensor algebra over the augmentation ideal of $u(\mathfrak h)$, i.e., $\Omega\, u(\mathfrak g):=\bigoplus_{i\ge 0} (u(\mathfrak h)^+)^{\otimes i}$.
\item[(ii)] The differentials of $\Omega\, u(\mathfrak h)$ are given by
\[
\partial^n=\sum_{i=0}^{n-1}(-1)^{i+1}\, 1^{\otimes i} \otimes \overline{\Delta}\otimes 1^{\otimes (n-i-1)},
\]
where $\overline{\Delta}(r)=\Delta(r)-r\otimes 1-1\otimes r$, for any $r\in u(\mathfrak h)^+$.
\end{itemize}
\end{deff}
Basic properties of bar and cobar constructions can be found in \cite[\S 19]{FHT}. The differentials $\partial^1$ and $\partial^2$ of $\Omega\, u(\mathfrak h)$ can be written explicitly in terms of $\Omega^1\, u(\mathfrak h)=u(\mathfrak h)^+$ and $\Omega^2\, u(\mathfrak h)=(u(\mathfrak h)^+)^{\otimes 2}$. For any $r,s\in u(\mathfrak h)^+$, we have 
\begin{align}\label{E:partials}
\begin{split}
\partial^1(r)&= 1\otimes r -\Delta(r) + r \otimes 1,\\
\partial^2(r\otimes s)&= 1\otimes r \otimes s -\Delta(r)\otimes s + r\otimes\Delta(s)- r\otimes s\otimes 1.
\end{split}
\end{align} 
We will apply Equations~(\ref{E:partials}) repeatedly throughout the paper.

%%%%%%%%%%%%%%%%%%%%%%%%%%%%%%%%

\section{Primitive deformation theory}
\label{sec:PD theory}

In this section, we will introduce the concept of Primitive Deformation and study its basic properties. First, we will define abelian types $\T$ from an action of a certain abelian restricted Lie algebra onto another. 

\begin{deff}\label{D:ARRLie}
Let $\mathfrak h$ and $\mathfrak g$ be two restricted Lie algebras. An \emph{algebraic representation} of $\mathfrak g$ on $\mathfrak h$ is a linear map $\rho: \mathfrak g\to \End_\field(\mathfrak h)$ such that
\begin{itemize}
\item[(i)] $\rho_{[x,y]}=\rho_x\rho_y-\rho_y\rho_x$,
\item[(ii)] $\rho_{(x^{[p]})}=(\rho_x)^p$,
\item[(iii)] $\rho_x([a,b])=[\rho_x(a),b]+[a,\rho_x(b)]$,
\item[(iv)] $\rho_x(a^p)=\rho_x(a)(\ad \, a)^{p-1}$,
\end{itemize}
for any $x,y\in \mathfrak g$ and $a,b\in \mathfrak h$. Any linear map $\rho: \mathfrak g\to \End_\field(\mathfrak h)$ satisfying only (i) and (ii) is referred as a \emph{restricted Lie algebra representation} of $\mathfrak g$ on $\mathfrak h$; see \cite{Hoch54,Jac52}.
\end{deff}

Let $\mathfrak g,\mathfrak h$ be two restricted Lie algebras, and $\rho: \mathfrak g\to \End_\field(\mathfrak h)$ be a linear map. The \emph{semiproduct} of $\mathfrak h$ and $\mathfrak g$ via $\rho$, denoted by $\mathfrak h\rtimes_\rho \mathfrak g$, is defined as follows. As vector spaces, $\mathfrak h\rtimes_\rho \mathfrak g=\mathfrak h\oplus \mathfrak g$. Regarding the restricted Lie algebra structure on $\mathfrak h\rtimes_\rho \mathfrak g$, for any $a,b\in \mathfrak h$ and $x,y\in \mathfrak g$, we have 
\begin{itemize}
\item[(i)] the Lie bracket is given by $[a+x,b+y]:=[a,b]_\mathfrak h+\rho_x(b)-\rho_y(a)+[x,y]_\mathfrak g$;
\item[(ii)] the restricted map is given by $(a+x)^{[p]}=a^{[p]}+x^{[p]}+\sum_{i=1}^{p-1}s_i(a,x)$, where $is_i(a,x)$ is the coefficient of $\lambda^{i-1}$ in $a (\ad \, (\lambda a+x))^{p-1}$.
\end{itemize}

\begin{prop}\label{P:Type}
The semiproduct $\mathfrak h\rtimes_\rho \mathfrak g$ is a well-defined restricted Lie algebra if and only if $\rho$ is an algebraic representation of $\mathfrak g$ on $\mathfrak h$.
\end{prop}
\begin{proof}
We write $V$ for the semiproduct $\mathfrak h\rtimes_\rho \mathfrak g$. We first assume that $V$ is a well-defined restricted Lie algebra. For any $a,b\in \mathfrak h\subseteq V$, and $x,y\in \mathfrak g\subseteq V$, the Jacobi identity of the Lie bracket of $V$ implies that $[[x,y],a]=[x,[y,x]]-[y,[x,a]]$ and $[x,[a,b]]=[[x,a],b]+[a,[x,b]]$. By above definition of the Lie bracket of $V$, we get conditions (i) and (iii) in Definition \ref{D:ARRLie} for $\rho$. Moreover, by condition (iii) in Definition \ref{D:RLie} of a restricted Lie algebra, we have $[x^p,a]=(\ad \, x)^p (a)$ and $[x,a^p]=x (\ad \, a)^p=[x,a](\ad \, a)^{p-1}$. This implies conditions (ii) and (iv) in Definition \ref{D:ARRLie}. Hence, $\rho$ is an algebraic representation of $\mathfrak g$ on $\mathfrak h$.

Conversely, suppose $\rho$ is an algebraic representation of $\mathfrak g$ on $\mathfrak h$. Clearly, one sees that $V= \mathfrak h\rtimes_\rho \mathfrak g$ is a well-defined Lie algebra. We denote by $U$ its universal enveloping algebra. Consider the ideal $I$ in $U$ generated by $a^p-a^{[p]}$ and $x^p-x^{[p]}$ for every $a\in \mathfrak h$ and $x\in \mathfrak g$. It is an easy exercise, by Bergman's Diamond Lemma \cite{Beg}, to verify that $V$ can be embedded into the quotient algebra $U/I$. Moreover in $U/I$, by \cite[pp.~186--187]{Jac}, we have
$$(a+x)^p=a^p+x^p+\sum_{i=1}^{p-1}s_i(a,x)=a^{[p]}+x^{[p]}+\sum_{i=1}^{p-1}s_i(a,x),$$ 
where $is_i(a,x)$ is the coefficient of $\lambda^{i-1}$ in $a(\ad \, (\lambda a+x))^{p-1}$. Since each $s_i(a,x) \in V$, $V$ is invariant under the commutator and $p$-th power map of $U/I$, respectively. Moreover, the restricted Lie algebra structure defined on $V$ above is indeed induced correspondingly by the commutator and $p$-th power map of $U/I$. This shows that the restricted Lie algebra structure is well-defined on $V$; see Example \ref{E:RLA}. 
\end{proof}

\begin{deff}\label{D:Type}
An \emph{abelian type (of rank $n$)} $\T$ is a triple $(\mathfrak g,\mathfrak h,\rho)$ satisfying the following descriptions.
\begin{itemize}
\item[(i)] $\mathfrak g$ and $\mathfrak h$ are abelian restricted Lie algebras.
\item[(ii)] $\dim \mathfrak g=1$ and $\dim \mathfrak h=n$, for some integer $n\ge 1$.
\item[(iii)] $\rho$ is an algebraic representation of $\mathfrak g$ on $\mathfrak h$.
\end{itemize}
We call $n=\dim \mathfrak h$ the {\it rank} of $\T$. 
\end{deff}
As shown in \cite[Lemma 4.3]{XWang}, any abelian type $\T=(\mathfrak g,\mathfrak h,\rho)$ is associated with an abelian matched pair $(\rightharpoonup,\varrho)$ for Hopf algebra extensions of $u(\mathfrak h)$ by $u(\mathfrak g)$, see \cite[Definition 2.2]{AMa}. %consists of an action $\rightharpoonup: u(\mathfrak g)\otimes u(\mathfrak h)\to u(\mathfrak h)$ induced by $\rho$ and a trivial coaction $\varrho: u(\mathfrak g)\to u(\mathfrak g)\otimes u(\mathfrak h)$. 

\begin{deff}\label{D:2TIsom}
Two abelian types $\T=(\mathfrak g,\mathfrak h,\rho)$ and $\T'=(\mathfrak g',\mathfrak h',\rho')$ are said to be \emph{isomorphic} if there exist restricted Lie algebra isomorphisms $\phi_1: \mathfrak g'\to \mathfrak g$ and $\phi_2: \mathfrak h\to \mathfrak h'$ such that the following diagram commutes:
\begin{align*}%\label{D:IsoGraph}
\xymatrix{
\mathfrak g\otimes \mathfrak h\ar[rr]^-{\rho}\ar[d]_-{1\otimes \phi_2} &&\mathfrak h\ar[d]^-{\phi_2}\\
\mathfrak g\otimes \mathfrak h'\ar[r]_-{\phi_1^{-1}\otimes 1}&\mathfrak g'\otimes \mathfrak h'\ar[r]_-{\rho'}&\mathfrak h'.
}
\end{align*}
For simplicity, we will use $\phi$ to denote the isomorphism pair $(\phi_1,\phi_2)$, and we say that $\T$ is isomorphic to $\T'$ via $\phi$. 
\end{deff}

Let $\T=(\mathfrak g,\mathfrak h,\rho)$ be any abelian type of rank $n$. Recall that the restricted universal enveloping algebra of $\mathfrak h$, denoted by $u(\mathfrak h)$, is a commutative ($\mathfrak h$ is abelian) and cocommutative quantum $p$-group of dimension $p^{n}$. Next, we will discuss the cobar construction on $u(\mathfrak h)$, whose second cohomology group will be used for isomorphism classes in primitive deformation theory. 

In order to describe cohomology groups of $\Omega\, u(\mathfrak h)$, we will use the map $\Bock: \mathfrak h\to (u(\mathfrak h)^+)^{\otimes 2}$ defined as: 
\begin{align}\label{E:Bock}
\Bock(r)=\sum_{1\le i\le p-1}\frac{(p-1)!}{i!\,(p-i)!}\, \left(r^i\otimes r^{p-i}\right), %=-\sum_{\stackrel{i+j=p}{1\le i,j\le p-1}}\frac{r^i}{i!}\otimes \frac{r^j}{j!},
\end{align}
for every $r\in \mathfrak h$. It is easy to check that $\Bock(r)$ is a $2$-cocycle in $\Omega\, u(\mathfrak h)$ for every $r\in \mathfrak h$.

\begin{lem}\label{L:HochUH}
For the cohomology ring, we have 
\[\HL^*\left(\Omega\, u(\mathfrak h)\right)\cong 
\begin{cases} S(\mathfrak h), & \text{if } p=2 \\
 \Lambda (\mathfrak h)\otimes S\left(\Bock(\mathfrak h)\right), & \text{if } p > 2
 \end{cases} 
\]
where $\Lambda$ and $S$ are the exterior and polynomial algebra functors, respectively. 
\end{lem}
\begin{proof}
Let $A$ be the vector space dual of $u(\mathfrak h)$. The coalgebra structure of $u(\mathfrak h)$ equips $A$ with an algebra structure together with an augmentation by evaluating at the unit of $u(\mathfrak h)$. By above discussion, $A \cong \field[C_p^n]$, a $p$-group algebra. Clearly, there is an isomorphism of complexes between the cobar construction on $u(\mathfrak h)$ and the following reduced complex
\begin{align*}
\xymatrix{
\k \ar[r]^-{0} & \Hom_\field (A^+,\field) \ar[r]^-{\delta^1} & \Hom_\field((A^+)^{\otimes 2},\field)  \ar[r]^-{\delta^2} &\Hom_\field((A^+)^{\otimes 3},\field)\ar[r]^-{\delta^3}&\cdots.
}
\end{align*}
The differentials are given by $\delta^{j}(f) := \sum_{i=0}^{j-1}(-1)^{i+1}f\circ(1^{\otimes i} \otimes m \otimes 1^{\otimes {(j-i-1)}})$, where $m$ is the multiplication map of $A$ restricted to $A^+$ and $f: (A^+)^{\otimes j}\to \field$. As stated in \cite[Section 2.4]{MPSW}, the above complex calculates the Hochschild cohomology $\HL^*(A,\field)$, whose cup product is compatible with the product of the cohomology ring of $\Omega\,u(\mathfrak h)$ induced by the tensor product. Hence, $\HL^*\left(\Omega\, u(\mathfrak h)\right)\cong \HL^*(A,\field)\cong \HL^*(C_p^n,\field)$ as cohomology rings. It is easy to check that $\HL^1\left(\Omega\, u(\mathfrak h)\right)\cong \mathfrak h$. By \cite[Proposition 6.2]{wang2012connected}, we have $\HL^2\left(\Omega\, u(\mathfrak h)\right)\cong \Lambda^2(\mathfrak h)\oplus \Bock(\mathfrak h)$. Note that $  \Lambda^2(\mathfrak h) \oplus \Bock(\mathfrak h)=\mathfrak h \otimes \mathfrak h$ when $p=2$. Then the result follows from the well-known formula of the group cohomology for $C_p^n$; see \cite[Section 4]{DQuillen}.
\end{proof}

According to \cite[Definition 1.1 \& Proposition 1.4]{XWang}, for any abelian type $\T=(\mathfrak g,\mathfrak h,\rho)$, the algebraic representation $\rho$ gives a Hopf $u(\mathfrak g)$-module structure on $u(\mathfrak h)$. Moreover, the action of $\rho$ on $\mathfrak h$ can be extended to the cobar construction $\Omega\, u(\mathfrak h)$ by derivation \cite[Proposition 1.5]{XWang}. Since $\mathfrak g$ is 1-dimensional, we can specify $\rho$ by $\rho_z$, for any choice of nonzero $z\in \mathfrak g$. In the following, we still use $\rho$ to denote the map $\rho: \mathfrak g\to \Der(\Omega\, u(\mathfrak h))$ of restricted Lie algebras, given by $z \mapsto \rho_z$.

\begin{deff}\label{D:TypeChar}
Let $\T=(\mathfrak g,\mathfrak h,\rho)$ be any abelian type. Define $\Phi: \mathfrak g\to \End_{\mathbb F_p}(\Omega\, u(\mathfrak h))$ to be the $\mathbb F_p$-linear map from $\mathfrak g$ to the endomorphisms of $\Omega\, u(\mathfrak h)$ via 
\[
\Phi_z(\pi)=\pi^p-\lambda\, \pi+\rho_z^{p-1}(\pi),
\] 
preserving the grading, where $z\in \mathfrak g$ satisfying $z^{[p]}=\lambda z$, for some $\lambda \in \field$, and $\pi\in \Omega^*\, u(\mathfrak h)$ is a homogeneous element with $\pi^p$ denoting its $p$-th power in $\Omega^*\, u(\mathfrak h)$.
\end{deff}
In lower degrees of $\Omega\, u(\mathfrak h)$, $\Phi$ is explicitly given by:  
\begin{align} \label{E:phi}
\begin{split}
\Phi_z(r)&=r^p-\lambda\, r+\rho_z^{p-1}(r); \\
\Phi_z(r\otimes s)&=r^p\otimes s^p-\lambda\ (r\otimes s)+(\rho_z\otimes 1+1\otimes \rho_z)^{p-1}(r\otimes s)\\ 
&=r^p\otimes s^p-\lambda\ (r\otimes s)+\sum_{0\le i\le p-1}{p-1\choose i}\ \rho_z^{i}(r)\otimes \rho_z^{p-1-i}(s),
\end{split}
\end{align}
for any $r\in \Omega^1\, u(\mathfrak h)=u(\mathfrak h)^+$ and $r\otimes s\in \Omega^2\, u(\mathfrak h)=(u(\mathfrak h)^+)^{\otimes 2}$. Here, the first equality of $\Phi_z(r\otimes s)$ is based on the definition of $\Phi_z$ and the fact that $\rho_z$ is a Lie algebra map and hence preserves the coproduct structure of $u(\mathfrak h)$. We will apply Equations~(\ref{E:phi}) repeatedly throughout the paper. 

\begin{prop}\label{P:Cobar}
Let $\T=(\mathfrak g,\mathfrak h,\rho)$ be any abelian type, and $z\in \mathfrak g$. In $\Omega\, u(\mathfrak h)$: 
\begin{itemize} 
\item[(i)] $\rho_z$ and $\Phi_z$ commute with the differentials of $\Omega\, u(\mathfrak h)$;
\item[(ii)] $\rho_z\circ\Phi_z=0$;
\item[(iii)] $\displaystyle \rho_z^{i+1}(\Bock(r))=\partial^1 \left( -\sum_{i_1+\ldots+i_p=i}\ \frac{i!}{i_1!\cdots i_p!} \ \rho_z^{i_1}(r)\cdots \rho_z^{i_{p-1}}(r)\rho_z^{1+i_p}(r) \right)$, for $i\ge 0$;
\item[(iv)] $\Bock(\alpha\, r)=\alpha^p\ \Bock(r)$, for any $\alpha\in \field$;
\item[(v)] $\displaystyle \Bock(r)+\Bock(s)-\Bock(r+s)=\partial^1\left (\sum_{i=1}^{p-1} \frac{(p-1)!}{i!(p-i)!} \ r^is^{p-i}\right)$, for any $r,s \in \mathfrak h$.
\end{itemize}
\end{prop}
\begin{proof}
(i) follows from \cite[Propositions 1.5 \& 2.3 (a)]{XWang}. (ii) follows from \cite[Proposition 2.3 (c)]{XWang}. %Note that $\Omega\, u(\mathfrak h)$ is a differential graded algebra generated in degree one. Hence, it suffices to check that $\rho_z$ and $\Phi_z$ commute with $\partial^1$, which follows from \cite[Definition 1.1 \& Proposition 1.4]{XWang}. (ii) Since $\rho_z$ is a derivation on $\Omega\, u(\mathfrak h)$, for any $\pi\in \Omega^*\, u(\mathfrak h)$, we have
%\begin{align*}
%(\rho_z\circ\Phi_z)(\pi)&=\rho_z\left(\pi^p-\lambda\, \pi+\rho_z^{p-1}(\pi)\right)\\
%&=\rho_z(\pi^p)-\lambda\, \rho_z(\pi)+\rho_z^p(\pi)\\
%&=p\,\pi^{p-1}\rho_z(\pi)+\rho_{(z^{[p]})}(\pi)-\rho_{(\lambda z)}(\pi)\\
%&=\rho_{(z^{[p]}-\lambda z)}(\pi) = 0.
%\end{align*}
%The last two equalities are due to characteristic $p$ and relation $z^{[p]} = \lambda z$. 
(iii) The special case $\rho_z\left(\Bock(r)\right)=\partial^1\left(-r^{p-1}\rho_z(r)\right)$ is exactly \cite[Lemma 6.11]{XWang}. An easy induction yields the general formula. (iv) is clear and (v) follows from a direct computation. 
\end{proof}

By Lemma \ref{L:HochUH}, we know $\HL^1\left(\Omega\, u(\mathfrak h)\right)\cong \mathfrak h$ and $\HL^2\left(\Omega\, u(\mathfrak h)\right)\cong \Lambda^2(\mathfrak h)\oplus \Bock(\mathfrak h)$. Hence, $\Bock$ can be viewed as a map from $\HL^1\left(\Omega\, u(\mathfrak h)\right)$ to $\HL^2\left(\Omega u(\mathfrak h)\right)$ via $\Bock: [r]\mapsto [\Bock(r)]$ for every $r\in \mathfrak h$. Also by Proposition \ref{P:Cobar}(i), $\Phi_z$ induces an $\mathbb F_p$-linear map on $\HL^*(\Omega\, u(\mathfrak h))$ for any $z\in \mathfrak g$.

\begin{cor}\label{C:SLinear}
Let $\T=(\mathfrak g,\mathfrak h,\rho)$ be any abelian type. In the cohomology ring of $\Omega\, u(\mathfrak h)$, we have
\begin{itemize}
\item[(i)] $\Bock$ is a semilinear map \emph{w.r.t.} to the Frobenius map of $\field$.
\item[(ii)] For any $z\in \mathfrak g$, $\Bock(\mathfrak h)$ is $\Phi_z$-invariant in $\HL^2\left(\Omega\,u(\mathfrak h)\right)$. When $p>2$, $\Lambda^2(\mathfrak h)$ is also $\Phi_z$-invariant in $\HL^2\left(\Omega\,u(\mathfrak h)\right)$.
\end{itemize}
\end{cor}

\begin{proof}
(i) is directly from Proposition \ref{P:Cobar} (iv) and (v). (ii) By Proposition \ref{P:Cobar} (iii)-(v), for any $r \in \mathfrak h$:
\begin{align*}
&\ \Phi_z\left[\Bock(r)\right]\\
=&\ \left[\Bock(r)^p-\lambda\, \Bock(r)+\rho_z^{p-1}\left(\Bock(r)\right)\right]\\
=&\ \left[\Bock(r)^p-\lambda\, \Bock(r)\right]-\left[\partial^1\left(\sum_{i_1+\ldots+i_p=p-2}\frac{i!}{i_1!\cdots i_p!}\ \rho_z^{i_1}(r)\cdots \rho_z^{i_{p-1}}(r)\rho_z^{1+i_p}(r)\right)\right]\\
=&\ \left[\Bock(r^p-\lambda^{1/p}\, r)\right].
\end{align*}
Thus, $\Bock(\mathfrak h)$ is $\Phi_z$-invariant in $\HL^2\left(\Omega\,u(\mathfrak h)\right)$. When $p>2$, observe that for any $r \in \mathfrak h$, $[r \otimes r] = [-\partial^1(\frac{1}{2}r^2)]=0$ in $\HL^2\left(\Omega\, u(\mathfrak h)\right)$ and $\Lambda^2(\mathfrak h)$ does not coincide with $\Bock(\mathfrak h)$. By~Equation~(\ref{E:phi}), for any $r,s\in \mathfrak h$:
\[
\Phi_z[r\otimes s]=\left[r^p\otimes s^p-\lambda\, (r\otimes s)+\sum_{i=0}^{p-1}{p-1\choose i}\ \rho_z^i(r)\otimes \rho_z^{p-1-i}(s)\right]\in \Lambda^2(\mathfrak h).
\]
Hence, $\Lambda^2(\mathfrak h)$ is $\Phi_z$-invariant. This completes the proof.
\end{proof}

%In \cite{XWang}, quantum $p$-groups in $\mathscr X(n)$ arise as certain Hopf deformations of restricted universal enveloping algebras of abelian types with restriction on the resulting primitive space. Moreover, using primitive cohomology, isomorphism classes of quantum $p$-groups can be described as a cohomological group quotient by the automorphism group of each abelian type. 

In the following, we will introduce the concept called Primitive Deformation to formalize the construction of quantum $p$-groups in $\mathscr X(n)$. Let $\T=(\mathfrak g,\mathfrak h,\rho)$ be any abelian type of rank $n$. By Proposition \ref{P:Type}, we can construct a restricted Lie algebra $\mathfrak h\rtimes_\rho \mathfrak g$ by semiproduct, whose restricted universal enveloping algebra is denoted by $u(\T)$. Slightly abusing notation, we still use $\mathfrak h$ and $\mathfrak g$ for the generating space of $u(\T)$. Since $\mathfrak g$ is one dimensional, we can choose any $0\neq z\in \mathfrak g$ as a basis element. Then $u(\T)$ is isomorphic to the quotient algebra of the free $u(\mathfrak h)$-algebra, $u(\mathfrak h)\langle z\rangle$, subject to relations
\begin{align*}
[z,r]=\rho_z(r),\ z^p-z^{[p]}=0, 
\end{align*}
for any $r\in u(\mathfrak h)$, and $z^{[p]}$ denotes the restricted map of $\mathfrak g$ on $z$. One sees that $u(\T)$ is a quantum $p$-group of dimension $p^{n+1}$, where the coalgebra structure is given by:
\begin{gather*}
\Delta(x)=x\otimes 1+1\otimes x,\ \Delta(z)=z\otimes 1+1\otimes z,\\
\epsilon(x)=\epsilon(z)=0,
\end{gather*}
together with antipode $S$ given by $S(x)=-x$ and $S(z)=-z$, for any $x\in \mathfrak h$ and $z\in \mathfrak g$. Note that the primitive space of $u(\T)$ is $(n+1)$-dimensional spanned by $\mathfrak h$ and $\mathfrak g$. Based on $u(\T)$, we can construct quantum $p$-groups in $\mathscr X(n)$.

\begin{deff}\label{D:FirstPCD}
Let $\T=(\mathfrak g,\mathfrak h,\rho)$ be an abelian type of rank $n$, and $\Theta\in u(\mathfrak h)^+$ and $\chi\in (u(\mathfrak h)^+)^{\otimes 2}$ be two elements. We refer to the triple  $(\T,\Theta,\chi)$ as a datum and write $\mathscr D:=(\T,\Theta,\chi)$. For any $0\neq z\in \mathfrak g$, we can deform the algebra and coalgebra structures of $u(\T)$ regarding the element $z$ such that 
\[
z^p-z^{[p]}+\Theta=0\ \text{and}\ \Delta(z)=z\otimes 1+1\otimes z+\chi.
\]
We denote $u_z(\mathscr D)$ the correspondingly deformed quotient algebra of $u(\T)$ by above relations. Moreover, we say $u_z(\mathscr D)$ is a \emph{primitive deformation} (PD) of $u(\T)$ if it satisfies the following conditions:
\begin{itemize}
\item[(i)] $u_z(\mathscr D)$ is a quantum $p$-group of dimension $p^{n+1}$.
\item[(ii)] The primitive space of $u_z(\mathscr D)$ is isomorphic to $\mathfrak h$ of dimension $n$.
\end{itemize}
In such a case, $\mathscr D$ is said to be a \emph{PD datum} with respect to $z$.
\end{deff}

It directly follows from the definition that if $u_z(\mathscr D)$ is a PD of $u(\T)$ for some PD datum $\mathscr D=(\T,\Theta,\chi)$, then $u_z(\mathscr D)$ is a quantum $p$-group in $\mathscr X(n)$. On the other hand, any quantum $p$-group in $\mathscr X(n)$ can be realized as some PD of $u(\T)$ for some abelian type $\T$ of rank $n$ by \cite[Corollary 4.6]{XWang}. 

\begin{prop}\label{P:RCdata}
Let $\T=(\mathfrak g,\mathfrak h,\rho)$ be an abelian type, and $\mathscr D=(\T,\Theta,\chi)$ be a datum. Then $\mathscr D$ is a PD datum with respect to some $0\neq z\in\mathfrak g$ if and only if it satisfies the following conditions:
\begin{itemize}
\item[(i)] $\chi$ is a $2$-cocycle but not a $2$-coboundary in $\Omega\, u(\mathfrak h)$, that is $\partial^2(\chi)=0$ and $\chi\neq \partial^1(s)$ for any $s\in u(\mathfrak h)^+$.
\item[(ii)] $\rho_z(\Theta)=0$.
\item[(iii)] $\Phi_z(\chi)=\partial^1(\Theta)$.
\end{itemize}
\end{prop}

\begin{proof}
It follows from \cite[Theorem 0.4 (a), case $n=1$]{XWang} that any datum $\mathscr D=(\T, \Theta,\chi)$, which satisfies (ii), (iii) with $\chi$ being a $2$-cocycle, gives arise to a quantum $p$-group $u_z(\mathscr D)$ of dimension $p^{\dim \mathfrak h+1}$. Moreover, \cite[Theorem 0.4 (c)]{XWang} states that $u_z(\mathscr D)$ has primitive space isomorphic to $\mathfrak h$ if and only if $[\chi]\neq 0$ in $\HL^2\left(\Omega\, u(\mathfrak h)\right)$. This proves one direction.  

Conversely, suppose $u_z(\mathscr D)$ is a PD of $u(\T)$. By the coassociativity of $\Delta$ on $z$:
\begin{align*}
0=&\ (1\otimes \Delta)\Delta(z)-(\Delta\otimes 1)\Delta(z)\\
=&\ (1\otimes \Delta-\Delta\otimes 1)(z\otimes 1+1\otimes z+\chi)\\
=&\ z\otimes 1\otimes 1+1\otimes (z\otimes 1+1\otimes z+\chi)+(1\otimes \Delta)(\chi)-(z\otimes 1+1\otimes z+\chi)\otimes 1\\
&-1\otimes 1\otimes z-(\Delta\otimes 1)(\chi)\\
=&\ 1\otimes \chi-(\Delta\otimes 1)(\chi)+(1\otimes \Delta)(\chi)-\chi\otimes 1\\
=&\ \partial^2(\chi).
\end{align*}
So $\chi$ is a $2$-cocycle. Write $z^{[p]}=\lambda z$, for some $\lambda\in \field$, (ii) comes from Definition~\ref{D:FirstPCD}, $\rho_z(\Theta)=[z,\Theta]=-[z,z^p-\lambda z]=0$. For (iii), we apply \cite[Lemma 2.2]{XWang}, Equations~(\ref{E:partials}) and (\ref{E:phi}) to get 
\begin{align*}
0=&\ \Delta(z^p-\lambda z+\Theta)\\
=&\ \Delta(z)^p-\lambda\,\Delta(z)+\Delta(\Theta)\\
=&\ (z\otimes 1+1\otimes z+\chi)^p-\lambda\, (z\otimes 1+1\otimes z+\chi)+\Delta(\Theta)\\
=&\ z^p\otimes 1+1\otimes z^p+\chi^p+\rho_z^{p-1}(\chi)-\lambda(z\otimes 1+1\otimes z+\chi)+\Delta(\Theta)\\
=&\ \chi^p-\lambda\,\chi+\rho_z^{p-1}(\chi)+(z^p-\lambda z)\otimes 1+1\otimes (z^p-\lambda z)+\Delta(\Theta)\\
=&\ \chi^p-\lambda\,\chi+\rho_z^{p-1}(\chi)-(1\otimes \Theta-\Delta(\Theta)+\Theta\otimes 1)\\
=&\ \Phi_z(\chi)-\partial^1(\Theta).
\end{align*}
Finally, one sees that $\chi$ is not a $2$-coboundary by \cite[Theorem 0.4 (c)]{XWang} again. This completes the proof.
\end{proof}

\begin{remark}\label{R:CCzero}
Let $\mathscr D=(\T,\Theta,\chi)$ be a PD datum with respect to some $0\neq z\in \mathfrak g$. 
\begin{itemize}
\item[(i)]  We have $[\Phi_z(\chi)]=[\partial^1(\Theta)]=0$ in $\HL^2\left(\Omega\, u(\mathfrak h)\right)$. 
\item[(ii)] The antipode $S$ of $u_z(\mathscr D)$ is the same as the antipode of $u(\T)$ when restricted to $u(\mathfrak h)$ and $S(z)=-z-m(S\otimes \text{Id})(\chi)$, where $m$ is the multiplication of $u(\mathfrak h)$.
\end{itemize}
\end{remark}

%%%%%%%%%%%%%%%%%%%%%%%%%%%%%%%%

\section{Isomorphism classes of primitive deformation}
\label{sec:iso classes of PD}

In this section, we introduce a cohomological-type set together with an automorphism group action in order to describe isomorphism classes of primitive deformations of an abelian type. The following lemma is fundamental in the classification of such primitive deformations in terms of PD data.

\begin{lem}
Let $\T=(\mathfrak g,\mathfrak h,\rho)$ and $\T'=(\mathfrak g',\mathfrak h',\rho')$ be two abelian types, and let $z\in \mathfrak g$ and $z'\in\mathfrak g'$ be any two nonzero elements. 
\begin{itemize}
\item[(i)] Suppose $\mathscr D=(\T,\Theta,\chi)$ and $\mathscr D'=(\T',\Theta',\chi')$ are two PD data with respect to $z$ and $z'$, respectively. If $u_z(\mathscr D)\cong u_{z'}(\mathscr D')$ as Hopf algebras, then $\T\cong \T'$.
\item[(ii)] Suppose $\T\cong \T'$. For any PD datum $\mathscr D=(\T,\Theta,\chi)$ with respect to $z$, there exists a PD datum $\mathscr D'=(\T',\Theta',\chi')$ with respect to $z'$ such that $u_z(\mathscr D)\cong u_{z'}(\mathscr D')$ as Hopf algebras. 
\end{itemize}   
\end{lem}
\begin{proof}
(i) Let $F: u_z(\mathscr D)\to u_{z'}(\mathscr D')$ be an isomorphism of Hopf algebras. We apply similar idea as in the proof of \cite[Lemma 4.7]{XWang}. Since $F$ preserves the primitive space, we have the following commutative diagram of extensions of Hopf algebras:
\[
\xymatrix{
1\ar[rr]
&&u(\mathfrak h)\ar[rr]\ar[d]_-{F|_{u(\mathfrak h)}}
&&u_z(\mathscr D)\ar[rr]\ar[d]_-{F}
&&u(\mathfrak g)\ar[rr]\ar[d]_-{\overline{F}} 
&& 1\\
1\ar[rr]
&&u(\mathfrak h')\ar[rr]
&&u_{z'}(\mathscr D')\ar[rr]
&&u(\mathfrak g')\ar[rr]
&&1.
}
\]
By the definition of abelian matched pair associated to $u(\mathfrak g)$-extension over $u(\mathfrak h)$ (resp. $u(\mathfrak g')$-extension over $u(\mathfrak h')$) \cite[Definition \& Lemma 3.5]{Ho}, it is easy to check the pair of isomorphisms $(\overline{F},F|_{u(\mathfrak h)})$ induces an isomorphism between their corresponding matched pairs. Hence it gives an isomorphism between $\T$ and $\T'$ by \cite[Lemma 4.3 \& Corollary 4.4]{XWang}. (ii) Without loss of generality, we can assume that $\T\cong \T'$ via $\phi=(\phi_1,\phi_2)=(\gamma,\phi_2)$, where $\phi_1(z')= \gamma z$, and we take $\gamma =1$. Naturally extend $\phi_1$ and $\phi_2$ to $u(\mathfrak g')$ and $u(\mathfrak h)$ by Definition~\ref{D:2TIsom}. Then one sees that we can take $\Theta'=\phi_2(\Theta)$ and $\chi'=(\phi_2\otimes\phi_2)(\chi)$.
\end{proof}

\begin{remark}\label{rmk:isotype} Part (i) of the above lemma implies that primitive deformations of non-isomorphic types are non-isomorphic. Part (ii) implies that for each abelian type $\T$ up to isomorphism, its primitive deformation can be determined up to some choice of basis element $z \in \mathfrak g$. Therefore, together with the comment below Definition~\ref{D:FirstPCD}, every quantum $p$-group in $\mathscr X(n)$ determines an abelian type $\T$ of rank $n$ up to isomorphisms, and we say such quantum $p$-group is of type $\T$.
\end{remark} 

\begin{deff}\label{D:EquiR}
Let $\T=(\mathfrak g,\mathfrak h,\rho)$ be an abelian type. We define the set with respect to some nonzero element $z$ of $\mathfrak g$
$$
\mathcal H^2(\T):=\set[\Bigg]
{
\left[(\Theta,\chi)\right] \given
\begin{tabular}{l }
$(\T,\Theta,\chi)$ is a PD datum \emph{w.r.t.} $z$, and $(\Theta,\chi)\sim (\Theta',\chi')$ iff\\ 
$\Theta'-\Theta=\Phi_z(s)$ and $\chi'-\chi=\partial^1(s)$, for some $s\in u(\mathfrak h)^+$.\\
\end{tabular}
}. 
$$
It is easy to check that the relation $\sim$ defined above is an equivalence relation among all PD data with respect to $z$. By \cite[Proposition~5.5]{XWang}, the set $\mathcal H^2(\T)$ classifies quantum $p$-groups in $\mathscr X(n)$ as Hopf extensions of $u(\mathfrak h)$ by $u(\mathfrak g)$, with abelian matched pair given by $\T$ up to equivalent extensions. 
\end{deff}

Throughout, let $\T=(\mathfrak g,\mathfrak h,\rho)$ be any abelian type. We fix a basis $z$ of $\mathfrak g$ with $z^{[p]}=\lambda z$, for some $\lambda\in \field$. Proposition~\ref{P:UH2} will show that our definitions of $\mathcal H^2(\T)$ and its $\Aut(\T)$-quotient do not depend on the choice of $z$. Note that isomorphisms from $\T$ to itself form a group with the group multiplication given by composition of maps according to Definition \ref{D:2TIsom}. We denote the group by $\Aut(\T)$. For any $\phi\in\Aut(\T)$, we can write $\phi=(\phi_1,\phi_2)=(\gamma,\phi_2)$, where $\phi_1$ and $\phi_2$ are two automorphisms of the restricted Lie algebras $\mathfrak g$ and $\mathfrak h$, respectively, such that $\phi_1(z)=\gamma z$, for some $\gamma\in \field^\times$. It is easy to check that $\gamma\,\lambda=\gamma^p\,\lambda$, since $\phi_1$ is an automorphism of $\mathfrak g$.

\begin{lem}\label{L:Atype}
Let $\mathscr D=(\T,\Theta,\chi)$ be a datum, and $\phi=(\gamma,\phi_2)\in \Aut(\T)$. Then $\phi_2$ can be extended to a graded automorphism of $\Omega\, u(\mathfrak h)$ commuting with the differentials. Moreover,
\begin{itemize}
\item[(i)] $\rho_z(\phi_2(\Theta))=\gamma\, \phi_2(\rho_z(\Theta))$;
\item[(ii)] $\Phi_z(\gamma\,(\phi_2\otimes \phi_2)(\chi))=\gamma^p\, (\phi_2\otimes \phi_2)\, \Phi_z(\chi)$.
\end{itemize}
\end{lem}
\begin{proof}
It is clear that $\phi_2$ can be extended to a Hopf automorphism of $u(\mathfrak h)^+$. Note that the cobar construction $\Omega\, u(\mathfrak h)$ is the tensor algebra $\bigoplus_{i\ge 0} (u(\mathfrak h)^+)^{\otimes i}$. Then, $\phi_2$ can be extended to a graded automorphism of $\Omega\, u(\mathfrak h)$, which commutes with the differentials since the differentials are defined in terms of the comultiplication of $u(\mathfrak h)^+$. 

(i) By the commutative diagram in Definition \ref{D:2TIsom}, we have $\rho_z\circ \phi_2=\gamma\,\phi_2\circ \rho_z$ on $u(\mathfrak h)^+$. (ii) Using (i) and the fact that $\gamma\lambda=\gamma^p\lambda$, we have 
\begin{align*}
&\ \Phi_z\left(\gamma\, (\phi_2\otimes \phi_2)(\chi)\right)\\
=&\ \left(\gamma\, (\phi_2\otimes \phi_2)(\chi)\right)^p-\lambda\,\left(\gamma\, (\phi_2\otimes \phi_2)(\chi)\right)+\rho_z^{p-1}\left(\gamma\,(\phi_2\otimes \phi_2)(\chi)\right)\\
=&\ \gamma^p\, (\phi_2\otimes \phi_2)(\chi^p)-(\gamma\lambda)\, (\phi_2\otimes \phi_2)(\chi)+\gamma\left(\rho_z\otimes 1+1\otimes \rho_z\right)^{p-1}(\phi_2\otimes \phi_2)(\chi)\\
=&\ \gamma^p\,(\phi_2\otimes \phi_2)(\chi^p)-(\gamma^p\lambda)\, (\phi_2\otimes \phi_2)(\chi)+\gamma^p\,(\phi_2\otimes \phi_2)\left(\rho_z\otimes 1+1\otimes \rho_z\right)^{p-1}(\chi)\\
=&\ \gamma^p\,(\phi_2\otimes \phi_2)\left(\chi^p-\lambda\, \chi+\rho_z^{p-1}(\chi)\right)\\
=&\ \gamma^p\,(\phi_2\otimes \phi_2)\left(\Phi_z(\chi)\right).
\end{align*}
\end{proof}

\begin{lem}\label{L:PCDAct}
Let $\mathscr D=(\T,\Theta,\chi)$ be a datum, and $\phi=(\gamma,\phi_2)\in \Aut(\T)$. Then, with respect to $z$,
\begin{itemize}
\item[(i)] $\mathscr D$ is a PD datum if and only if $\mathscr D':=\left(\T, \gamma^p\,\phi_2(\Theta),\gamma\,(\phi_2\otimes \phi_2)(\chi)\right)$ is a PD datum.
\item[(ii)] $\Aut(\T)$ acts on $\mathcal H^2(\T)$ induced by the map $(\Theta,\chi)\mapsto (\gamma^p\,\phi_2(\Theta),\gamma\, (\phi_2\otimes \phi_2)(\chi))$.
\end{itemize} 
\end{lem}
\begin{proof}
(i) It is clear that $\mathscr D'$ is a datum by Lemma \ref{L:Atype}. Assuming $\mathscr D$ is a PD datum with respect to $z$, we will show that so is $\mathscr D'$, that is, the conditions in Proposition \ref{P:RCdata} are satisfied. Note that
\begin{align*}
\partial^2\left(\gamma\,(\phi_2\otimes \phi_2)(\chi)\right)=\gamma\, \partial^2\left((\phi_2\otimes \phi_2)(\chi)\right)=\gamma\,(\phi_2\otimes \phi_2\otimes \phi_2)\partial^2(\chi)=0.
\end{align*}
Hence $\gamma\, (\phi_2\otimes \phi_2)(\chi)$ is a $2$-cocycle. If $\gamma\,(\phi_2\otimes \phi_2)(\chi)=\partial^1(s)$ for some $s\in u(\mathfrak h)^+$, then $ \chi=\partial^1(\gamma^{-1}\,\phi^{-1}_2(s))$. This contradicts to the fact that $\chi$ is not a $2$-coboundary. Thus, $\mathscr D'$ satisfies condition (i) in Proposition \ref{P:RCdata}. Next, by Lemma \ref{L:Atype} (i), we have
\[
\rho_z(\gamma^p\,\phi_2(\Theta))=\gamma^p\,\rho_z(\phi_2(\Theta))=\gamma^{p+1}\,\phi_2(\rho_z(\Theta))=0.
\]
Moreover, by Lemma \ref{L:Atype} (ii),
\begin{align*}
\Phi_z(\gamma\,(\phi_2\otimes \phi_2)(\chi))=\gamma^p\,(\phi_2\otimes \phi_2)(\Phi_z(\chi))=\gamma^p\,(\phi_2\otimes \phi_2)(\partial^1(\Theta))=\partial^1(\gamma^p\,\phi_2(\Theta)).
\end{align*}
Hence, $\mathscr D'$ is a PD datum with respect to $z$. The converse part can be proven by replacing $\phi$ with $\phi^{-1}=(\gamma^{-1},\phi_2^{-1})$. 

(ii) It follows from (i) and Definition \ref{D:FirstPCD} that the group $\Aut(\T)$ acts on the set $\{(\Theta,\chi)\,|\,(\T,\Theta,\chi)\ \text{is a PD datum {\it w.r.t.}}\, z\}$ via $\phi: (\Theta,\chi)\mapsto (\gamma^p\,\phi_2(\Theta),\gamma\, (\phi_2\otimes \phi_2)(\chi))$. Suppose $(\Theta,\chi)\sim (\Theta',\chi')$ by some $s\in\mathscr u(\mathfrak h)^+$, i.e.,
\begin{gather*}
\chi'-\chi=\partial^1(s)\ \mbox{and}\ \Theta'-\Theta=\Phi_z(s).
\end{gather*}
Similar computation as in (i) yields that  
\begin{align*}
\gamma^p\,\phi_2(\Theta')-\gamma^p\,\phi_2(\Theta)&=\Phi_z\left(\gamma\,\phi_2(s)\right),\\
 \gamma\, (\phi_2\otimes \phi_2)(\chi')-\gamma\,(\phi_2\otimes \phi_2)(\chi)&=\partial^1\left(\gamma\,\phi_2(s)\right).
\end{align*}
By Definition \ref{D:EquiR}, $(\gamma^p\,\phi_2(\Theta),\gamma\,(\phi_2\otimes \phi_2)(\chi))\sim (\gamma^p\,\phi_2(\Theta'),\gamma\,(\phi_2\otimes \phi_2)(\chi'))$, which implies that $\phi$ maps equivalence classes to equivalence classes in $\mathcal H^2(\T)$. Therefore, the $\Aut(\T)$-action on $\mathcal H^2(\T)$ is well-defined. 
\end{proof}

\begin{prop}\label{P:UH2}
Let $\T=(\mathfrak g,\mathfrak h,\rho)$ be any abelian type. Then $\mathcal H^2(\T)$ and its quotient $\mathcal H^2(\T)/\Aut(\T)$ are unique up to isomorphisms with respect to any $0\neq z\in \mathfrak g$.
\end{prop}
\begin{proof}
Let $z$ and $z'$ be two different bases of $\mathfrak g$. We write 
\[
z^{[p]}=\lambda\,z,\ z'^{[p]}=\lambda'\, z',\ z'=\gamma z
\]
for some $\lambda,\lambda'\in \field$ and $\gamma \in \field^\times$. Then we have $\gamma^p\lambda=\gamma\lambda'$. Suppose $\mathscr D=(\T,\Theta,\chi)$ is a PD datum with respect to $z$. By Proposition \ref{P:RCdata}, we have $\chi$ is a $2$-cocycle but not a $2$-coboundary in $\Omega\, u(\mathfrak h)$ and $\rho_z(\Theta)=0$ and $\Phi_z(\chi)=\partial^1(\Theta)$. Set the new datum $\mathscr D'=(\T,\gamma^p\Theta,\gamma\chi)$. By Definition \ref{D:TypeChar}, we have
\begin{align*}
\Phi_{z'}(\gamma\, \chi)&=(\gamma\,\chi)^p-\lambda'(\gamma\,\chi)+\rho_{z'}^{p-1}(\gamma\,\chi)\\
&=\gamma^p\,\chi^p-\gamma^p\lambda\, \chi+\gamma^p\, \rho_z^{p-1}(\chi)\\
&=\gamma^p\left(\chi^p-\lambda \chi+\rho_z^{p-1}(\chi)\right)\\
&=\gamma^p\Phi_z(\chi)\\
&=\partial^1(\gamma^p\,\Theta).
\end{align*}
Moreover, it is clear that $\rho_{z'}(\gamma^p\Theta)=0$ and $\gamma\chi$ is a $2$-cocycle but not a $2$-coboundary in $\Omega\, u(\mathfrak h)$. By Proposition~\ref{P:RCdata}, $\mathscr D'$ is a PD datum with respect to $z'$. Therefore, the map $(\Theta,\chi)\mapsto (\gamma^p \Theta,\gamma \chi)$ induces a map from $\mathcal H^2(\T)$ with respect to $z$ to $\mathcal H^2(\T)$ with respect to $z'$. One can check that such map is compatible with the $\Aut(T)$-action and it is a bijection with inverse map given by replacing $\gamma$ with $\gamma^{-1}$.
\end{proof}

                                                                                                                                                                                                                                                                                                                                                                                                                                                                                                                                                                                                                                                                                                                                                                                                                                                                                                                                                                                                                                                                                                                                                                                                                                                                                                                                                                                                                                                                                                                                     \begin{deff}\label{D:AdPer}                                                                                                                                                                                                                                                                                                                                                                                                                                                                                                                                                                                                                                                                                                                                                                                                                                                                                                                                                                                                                                                                                                                                                                                                                                                                                                                                                                                                                                                                                                                                 Let $\T=(\mathfrak g,\mathfrak h,\rho)$ be an abelian type. Suppose $0\neq z\in \mathfrak g$. By Proposition \ref{P:Cobar} (ii), $\Ker\, \rho_z\supseteq \Img\, \Phi_z$ on $\mathfrak h$. We say $\T$ is \emph{permissible} with respect to $z$ if $\Ker\, \rho_z=\Img\, \Phi_z$ on $\mathfrak h$, that is, whenever $\rho_z(s)=0$ for some $s\in \mathfrak h$, there exists an $r\in \mathfrak h$ such that $\Phi_z(r)=s$.
\end{deff}

By the following lemma, the definition of a permissible abelian type $\T=(\mathfrak g,\mathfrak h,\rho)$ does not depend on the choice of the basis $z$ of $\mathfrak g$. So we can say $\T$ is permissible without specifying the basis of $\mathfrak g$. 

\begin{lem}\label{L:PerT}
Let $\T=(\mathfrak g,\mathfrak h,\rho)$ be an abelian type, and let $z$ and $z'$ be two different bases of $\mathfrak g$. Then $\Ker\, \rho_z=\Img\, \Phi_z$ on $\mathfrak h$ if and only if $\Ker\, \rho_{z'}=\Img\, \Phi_{z'}$ on $\mathfrak h$. 
\end{lem}
\begin{proof}
It suffices to show that $\Ker\, \rho_{z'}=\Img\, \Phi_{z'}$ if $\Ker\, \rho_z=\Img\, \Phi_z$ on $\mathfrak h$. Note that $\Ker\, \rho_{z'}\supseteq \Img\, \Phi_{z'}$ on $\mathfrak h$ by Proposition \ref{P:Cobar} (ii). We write $z'=\gamma z$, for some $\gamma \in \field^\times$. Now let $s \in \Ker\, \rho_{z'}$, that is, $s \in \mathfrak h$ and $\rho_{z'}(s)=0$. Clearly, $\rho_z(\gamma^{-p}s)=0$. Hence by assumption, there exists $r\in \mathfrak h$ such that $\Phi_z(r)=\gamma^{-p}s$. Use the same calculation in Proposition \ref{P:UH2}, we get: 
$\Phi_{z'}(\gamma\, r)=\gamma^p\, \Phi_z(r)=\gamma^p\gamma^{-p}s=s$. So $s \in \Img \,\Phi_{z'}$, which shows the other side of the inclusion. 
%This completes the proof. 
\end{proof}

\begin{prop}
Let $\T=(\mathfrak g,\mathfrak h,\rho)$ be an abelian type. If one of the following holds
\begin{itemize}
\item[(i)] $\mathfrak h$ is a torus;
\item[(ii)] $\mathfrak h$ is $p$-nilpotent and $\mathfrak g$ is a torus,
\end{itemize}
then $\T$ is permissible.
\end{prop}
\begin{proof}
By Lemma \ref{L:PerT}, we can choose any basis $z$ of $\mathfrak g$. Write $z^{[p]}=\lambda\, z$ for some $\lambda\in \field$. (i) Assume $\mathfrak h$ is a torus. By \cite{Hoch54}, we know $[p]: \mathfrak h\to \mathfrak h$ is surjective. Hence $\rho_z=0$ according to Definition \ref{D:ARRLie} (iv). So $\Ker \rho_z=\mathfrak h$. Now $\Phi_z(r)=r^p-\lambda r+\rho_z^{p-1}(r)=r^p-\lambda r$, for any $r\in \mathfrak h$. One sees easily that $\Img \Phi_z=\mathfrak h$ as the base field $\field$ is algebraically closed. (ii) Assume $\mathfrak h$ is $p$-nilpotent and $\mathfrak g$ is a torus. Without loss of generality, we can take $\lambda=1$. It is easy to check that if $\rho_z(r)=0$ for some $r\in \mathfrak h$, then $\Phi_z(-r)=(-r)^p-(-r)+\rho_z^{p-1}(-r)=r$, which completes the proof.
\end{proof}

\begin{thm}\label{TH:HopfO}
For any abelian type $\T$ of rank $n$, the following are naturally bijective.
\begin{itemize}
\item[(i)] The set of isomorphism classes of quantum $p$-groups of type $\T$ in $\mathscr X(n)$.
\item[(ii)] The set of isomorphism classes of primitive deformations of $u(\T)$.
\item[(iii)] $\mathcal H^2(\T)/\Aut(\T)$.
\item[(iv)] $\mathscr A^+(\T)/\Aut(\T)$, provided $\T$ is permissible; cf.~Definition~\ref{D:AdPer}. 
\end{itemize}
\end{thm}
\begin{proof}
The bijection between (i) and (ii) comes from Definition \ref{D:FirstPCD} and the comment below it. The $\Aut(\T)$-action on $\mathcal H^2(\T)$ was given in Lemma \ref{L:PCDAct}. It follows from \cite[Theorem~0.6]{XWang} that (i) and (iii) are bijective. Finally, by later Theorem \ref{T:BPT}, (iii) and (iv) are bijective whenever $\T$ is permissible. 
\end{proof}

%%%%%%%%%%%%%%%%%%%%%%%%%%%%%%%%

\section{Realization of isomorphism classes in affine spaces}
\label{sec:realization}

Throughout this section, we assume the characteristic of $\field$ is $p>2$. We will set, for a permissible abelian type $\T$ of rank $n$, a bijection between $\mathcal H^2(\T)$ and a subset of the affine space $\mathbb A^{n(n+1)/2}_\field$, which is compatible with the $\Aut(\T)$-action. In the remaining of the section, let $\T=(\mathfrak g,\mathfrak h,\rho)$ be an abelian type of rank $n$. We fix bases $z$ of $\mathfrak g$ and $\{x_i\}_{1\le i\le n}$ of $\mathfrak h$. Note that the abelian restricted Lie algebras $\mathfrak g$ and $\mathfrak h$ are determined by their restricted maps on the bases. Then, the restricted maps of $\mathfrak g$ and $\mathfrak h$ can be denoted by a scalar $\lambda$ and a $n\times n$ matrix $R$, respectively, such that 
\[
z^{[p]}=\lambda z,\ \mbox{and}\
\begin{pmatrix}
x_1^{[p]}\\ \vdots \\ x_n^{[p]}
\end{pmatrix}
=R
\begin{pmatrix}
x_1\\ \vdots \\ x_n
\end{pmatrix}
.\]
Since $\mathfrak g$ is spanned by $z$, any algebraic representation $\rho$ of $\mathfrak g$ on $\mathfrak h$ is uniquely determined by $\rho_z\in \End_\k(\mathfrak h)$, which can be realized by a $n\times n$ matrix $M$ such that 
\[
\rho_z\begin{pmatrix} x_1\\ \vdots \\ x_n\end{pmatrix}
=M \begin{pmatrix} x_1\\ \vdots \\ x_n\end{pmatrix}.
\]
By Definition \ref{D:Type} (iii), it holds that $\rho_z(x_i^{[p]})=0$, for all $1\le i\le n$, and $(\rho_z)^p=\rho_{(z^{[p]})}=\rho_{(\lambda z)}=\lambda \rho_z$.
In terms of matrices, these conditions are equivalent to %if we write $(x^p,y^p)^T=R(x,y)^T$, then
\[
M\begin{pmatrix} x_1^{[p]}\\ \vdots \\ x_n^{[p]}\end{pmatrix}=MR\begin{pmatrix} x_1\\ \vdots \\ x_n\end{pmatrix}=0,\ \mbox{and}\ M^p=\lambda M.
\]
Therefore, we have the following lemma.

\begin{lem}
Any abelian type $\T=(\mathfrak g,\mathfrak h,\rho)$ can be determined by a triple $(\lambda,R,M)$ satisfying $MR=0$ and $M^p=\lambda M$.
\end{lem}

Now for any $\phi\in\Aut(\T)$, by Definition \ref{D:2TIsom}, we will identify $\phi=(\phi_1,\phi_2)$ with $(\gamma,\phi_2)$ or $(\gamma,G)$ such that 
\begin{align}\label{E:RepA}
\phi_1(z)=\gamma\, z,\ \text{and}\ \phi_2\begin{pmatrix} x_1\\ \vdots \\ x_n\end{pmatrix}=G\begin{pmatrix} x_1\\ \vdots \\ x_n\end{pmatrix},
\end{align}
for some $\gamma \in \field^\times$ and some $n\times n$ invertible matrix $G$. It is worthy to point out that for $\phi=(\gamma,G)$ to be an automorphism of $\T=(\lambda,R,M)$, the sufficient and necessary conditions are  
\begin{align}\label{E:RepC}
\widehat{G}R=RG,\ \gamma^p\lambda=\gamma\lambda,\ M=\gamma G^{-1}MG,
\end{align}
where $\widehat{G}$ is constructed from $G$ by taking $p$-th power of every entry of $G$. We also point out that the composition of two automorphisms $\phi,\phi'\in \Aut(\T)$ is given by 
\begin{align}\label{E:ComG}
\phi'\circ \phi=(\gamma'\gamma,\phi_2'\circ \phi_2)=(\gamma'\gamma,GG'),
\end{align}
since matrices act on bases by left multiplication. The following lemma is straightforward using the Equation \eqref{E:ComG} above. 

\begin{lem}\label{L:AutAction}
There is a group embedding from $\Aut(\T)$ into $\text{GL}_{n(n+1)/2}(\field)$ via, for any $P=(a_{ij},b_k)_{1\le i<j\le n,1\le k\le n}\in \mathbb A_\field^{n(n+1)/2}$ and $\phi=(\gamma,G)\in \Aut(\T)$,
\begin{align*}
\phi(a_{ij})=\left(\sum_{1\le r<s\le n} \gamma \, \text{det} \begin{bmatrix}G_{ri}& G_{rj}\\ G_{si} & G_{sj}\end{bmatrix} \right)a_{ij}, \quad and\ \phi(b_k)=\sum_{1\le s\le n}\gamma^{1/p}\,G_{sk}\, b_s.
\end{align*}
\end{lem}

By Remark~\ref{rmk:isotype}, isomorphic quantum-$p$ groups in $\mathscr X(n)$ come from isomorphic abelian types. %It is proved in \cite[Section 4]{XWang} the every abelian type $T=(\mathfrak h,\mathfrak g,\rho)$ is indeed an abelian matched pair \cite[Definition-Lemma 1.3]{AMa} for Hopf algebra extensions of $u(\mathfrak h)$ by $u(\mathfrak g)$. Moreover, every quantum $p$-group in $\mathscr X(n)$ arises from such extension given by some type $\T$, which is invariant under equivalent extensions. 
Hence, in order to classify $\mathscr X(n)$, we can first classify all possible types in the following sense. 

\begin{lem}\label{L:IsomCT}
Suppose two types $\T=(\mathfrak g,\mathfrak h,\rho)$ and $\T'=(\mathfrak g,\mathfrak h,\rho')$ are given by triples $(\lambda,R,M)$ and $(\lambda, R,M')$, respectively. Then, $\T$ is isomorphic to $\T'$ if and only if there exist an invertible $n\times n$ matrix $G$ and a nonzero scalar $\gamma$ such that 
\begin{align*}
\widehat{G}R=RG,\ \gamma^p\lambda=\gamma\lambda,\ M'=\gamma G^{-1}MG,
\end{align*}
where $\widehat{G}$ is constructed from $G$ by taking $p$-th power of every entry of $G$.  
\end{lem}
\begin{proof}
Suppose $\T$ is isomorphic to $\T'$ via some $\phi=(\phi_1,\phi_2)$ as given in Equation \eqref{E:RepA}. It is direct to check that the commutativity of the diagram in Definition \ref{D:2TIsom} is equivalent to the following matrix multiplication identity 
\begin{align*}
M'=\gamma G^{-1}MG.
\end{align*}
It remains to show the conditions on $G,\gamma$ for them to be automorphisms of $\mathfrak h$ and $\mathfrak g$, respectively. Because $\phi_2$ preserves the restricted map of $\mathfrak h$, one sees that $\phi_2(x_i^{[p]})=\phi_2(x_i)^{[p]}$, for all $1\le i\le n$, which is the same as $\widehat{G}R=RG$. One can similarly obtain $\gamma^p\lambda=\gamma\lambda$ by using the fact that $\phi_1$ is an automorphism of $\mathfrak g$. Since the if part is clear, we complete the proof. 
\end{proof}

\begin{remark}\label{R:H2A3}
Recall Lemma \ref{L:HochUH} that $\HL^2\left(\Omega\, u(\mathfrak h)\right)$ is spanned by $[x_i\otimes x_j]$ and $[\Bock(x_k)]$, for $1\le i<j\le n$ and $1\le k\le n$. For any point $P=(a_{ij},b_k)_{1\le i<j\le n,1\le k\le n}\in \mathbb A^{n(n+1)/2}_\field$, set 
\begin{align}\label{E:Cohomology}
\chi_P:=\sum_{1\le i<j\le n}a_{ij}\, x_i\otimes x_j+\Bock\left(\sum_{1\le k\le n} b_k\,x_k\right)
\end{align}
as a $2$-cocycle in $\Omega\, u(\mathfrak h)$. Hence there is a bijection between $\HL^2(\Omega\, u(\mathfrak h))$ and $\mathbb A^{n(n+1)/2}_\field$ via $[\chi_P]\longleftrightarrow P$. Moreover, nonzero cohomology classes correspond to nonzero points in $\mathbb A^{n(n+1)/2}_\field$.
\end{remark}

\begin{deff}\label{D:AdPer}
A point $P\in \mathbb A^{n(n+1)/2}_\field$ is said to be \emph{admissible} \emph{w.r.t.} the type $\T$ if there exists some $s\in u(\mathfrak h)^+$ such that $\Phi_z(\chi_P)=\partial^1(s)$ and $\rho_z(s)=0$. 
We denote by $\mathscr A^+(\T)$ the set of all nonzero admissible elements in $\mathbb A_\field^{n(n+1)/2}$. 
\end{deff}

\begin{remark}\label{R:APCD}
A point $P\in \mathbb A^{n(n+1)/2}_\field\setminus\{0\}$ is admissible \emph{w.r.t.} a type $\T$ if and only if there exists a PD datum $\mathscr D=(\T,s,\chi_P)$ for some $s\in u(\mathfrak h)^+$ with respect to $z$.
\end{remark}

Next, we establish the $\Aut(\T)$-action on $\mathscr A^+(\T)$. Recall Lemma \ref{L:AutAction} that $\Aut(\T)$ acts on $\mathbb A_\field^{n(n+1)/2}$ by, for any $\phi=(\gamma,G)\in \Aut(\T)$ and $P=(a_{ij},b_k)_{1\le i,j\le n,1\le k\le n}\in \mathbb A^{n(n+1)/2}_\field$,  
\begin{gather*}
\phi(P)=\left(\widetilde{a_{ij}},\widetilde{b_k}\right)_{1\le i,j\le n,1\le k\le n},\\
\text{where } \widetilde{a_{ij}}=\sum_{1\le r<s\le n}\gamma(G_{ri}G_{sj}-G_{rj}G_{si})\, a_{ij},\ \text{and } \widetilde{b_k}=\sum_{1\le s\le n}\gamma^{1/p}\,G_{sk}\, b_s.
\end{gather*}

\begin{lem}\label{L:ActionAT}
Let $P=(a_{ij},b_k)_{1\le i,j\le n,1\le k\le n}\in \mathbb A^{n(n+1)/2}_\field\setminus\{0\}$ be admissible, and $\mathscr D=(\T,\Theta,\chi_P)$ be a PD datum with respect to $z$. For any $\phi=(\gamma,\phi_2)=(\gamma,G)\in \Aut(T)$, set the element
\[\sigma:=\frac{1}{2}\sum_{1\le r<s\le n}\gamma\,a_{rs}\left(\sum_{1\le i<j\le n}2G_{ri}G_{sj}\,x_ix_j+\sum_{1\le i\le n}G_{ri}G_{si}\, x_i^2\right)\]
in $u(\mathfrak h)^+$. Then the following two classes are equivalent in $\mathcal H^2(\T)$
\begin{align*}
\left[(\gamma^p\, \phi_2(\Theta),\gamma\,(\phi_2\otimes \phi_2)(\chi_P))\right]=\left[(\gamma^p\, \phi_2(\Theta)+\Phi_z(\sigma),\chi_{\phi(P)})\right].
\end{align*}
\end{lem}

\begin{proof}
We keep the above notations and definitions of $\widetilde{a_{ij}}, \widetilde{b_k}$. We first claim that $\chi_{\phi(P)}-\gamma\, (\phi_2\otimes \phi_2)(\chi_P)=\partial^1(\sigma)$: 
\begin{align*}
\chi_{\phi(P)}&-\gamma\,(\phi_2\otimes \phi_2)(\chi_P)=\sum_{1\le i<j\le n}\widetilde{a_{ij}}\,x_i\otimes x_j+\Bock\left(\sum_{1\le k\le n}\widetilde{b_k}\,x_k\right)\\
&-\sum_{1\le i<j\le n}\gamma\,a_{ij}\,\phi_2(x_i)\otimes \phi_2(x_j)-\gamma\, \Bock\left(\sum_{1\le k\le n}\phi_2(b_k\,x_k)\right)\\
&=\sum_{1\le i<j\le n}\widetilde{a_{ij}}\,x_i\otimes x_j-\sum_{1\le i<j\le n}\gamma\,a_{ij}\,\left(\sum_{1\le r,s\le n}G_{ir}G_{js}\, x_r\otimes x_s\right)\\
&+\Bock\left(\sum_{1\le k,s\le n}\gamma^{1/p}b_kG_{ks}\,x_s\right)-\Bock\left(\sum_{1\le k,s\le n}\gamma^{1/p}b_kG_{ks}\,x_s\right)\\
&=\sum_{1\le i<j\le n}\widetilde{a_{ij}}\,x_i\otimes x_j-\sum_{1\le r<s\le n}\gamma\,a_{rs}\left(\sum_{1\le i<j\le n}G_{ri}G_{sj}\,x_i\otimes x_j\right.\\
&\left.+\sum_{1\le j<i\le n}G_{ri}G_{sj}\,x_i\otimes x_j+\sum_{1\le i\le n}G_{ri}G_{si}\, x_i\otimes x_i\right)\\
&=\sum_{1\le i<j\le n}\widetilde{a_{ij}}\,x_i\otimes x_j-\sum_{1\le r<s\le n}\gamma\,a_{rs}\left(\sum_{1\le i<j\le n}G_{ri}G_{sj}\,(x_i\otimes x_j+x_j\otimes x_i)\right.\\
&\left.-\sum_{1\le i \le n}G_{ri}G_{si}\, x_i\otimes x_i\right)\\
&=-\sum_{1\le r<s\le n}\gamma\,a_{rs}\left(\sum_{1\le i<j\le n}G_{ri}G_{sj}\, (x_i\otimes x_j+x_j\otimes x_i)+\sum_{1\le i\le n}G_{ri}G_{si}\, x_i\otimes x_i\right)
\end{align*}
Then the claim follows from applying the definition of $\sigma$ and the identity
\[
\partial^1(x_ix_j)=1\otimes (x_ix_j)-\Delta(x_ix_j)+(x_ix_j)\otimes 1=-(x_i\otimes x_j+x_j\otimes x_i)
\]
to the last equation above. 

Secondly, we claim that $\mathscr D':=(\T,\gamma^p\, \phi_2(\Theta)+\Phi_z(\sigma),\chi_{\phi(P)})$ is also a PD datum with respect to $z$. It is clear that Proposition \ref{P:RCdata} (i) holds for $\mathscr D'$ since $\phi(P)\neq 0$. For Proposition \ref{P:RCdata} (iii),
\begin{align*}
\Phi_z\left(\chi_{\phi(P)}\right)=&\ \Phi_z\left(\gamma\,(\phi_2\otimes \phi_2)(\chi_P)+\partial^1(\sigma)\right)\\
=&\ \Phi_z\left(\gamma\, (\phi_2\otimes \phi_2)(\chi_P)\right)+\Phi_z\left(\partial^1(\sigma)\right)\\
=&\ \gamma^p\,(\phi_2\otimes \phi_2)(\Phi_z(\chi_P))+\partial^1\left(\Phi_z(\sigma)\right)
\end{align*}
by Lemma \ref{L:Atype} (ii) and Proposition \ref{P:Cobar} (i). Using the fact that $\mathscr D=(\T,\Theta,\chi_P)$ is a PD datum and by Lemma \ref{L:Atype},
\begin{align*}
\Phi_z\left(\chi_{\phi(P)}\right)=\gamma^p\,(\phi_2\otimes \phi_2)(\partial^1(\Theta))+\partial^1\left(\Phi_z(\sigma)\right)=\partial^1\left(\gamma^p\,\phi_2(\Theta)+\Phi_z(\sigma)\right).
\end{align*}
One sees that $\mathscr D'$ satisfies Proposition \ref{P:RCdata} (ii), i.e., 
\begin{align*}
\rho_z\left(\gamma^p\,\phi_2(\Theta)+\Phi_z(\sigma)\right)=\gamma^{p+1}\,\phi_2(\rho_z(\Theta))+(\rho_z\circ\Phi_z)(\sigma)=0,
\end{align*}
which proves the claim. 

At last, the result follows from Definition \ref{D:EquiR} such that $\chi_{\phi(P)}-\gamma\, (\phi_2\otimes \phi_2)(\chi_P)=\partial^1(\sigma)$ and $(\gamma^p\, \phi_2(\Theta)+\Phi_z(\sigma))-\gamma^p\, \phi_2(\Theta)=\Phi_z(\sigma)$. This completes the proof. 
\end{proof}

\begin{cor}
The subset $\mathscr A^+(\T)$ of $\mathbb A_\field^{n(n+1)/2}$ is invariant under the $\Aut(\T)$-action.
\end{cor}
\begin{proof}
Let $P\in \mathscr A^+(\T)$, and $\phi\in \Aut(\T)$. It suffices to show that $\phi(P)\in \mathscr A^+(\T)$, which follows from Lemma \ref{L:ActionAT} and Remark \ref{R:APCD}.
\end{proof}

\begin{lem}\label{L:BIH2A2}
Let $\T=(\mathfrak g,\mathfrak h,\rho)$ be a permissible abelian type. Then, there is a bijection between $\mathcal H^2(\T)$ and $\mathscr A^+(\T)$ induced by $\xymatrix{
\left[(\Theta,\chi)\right]\ar@{<->}[r]&P,
}
$
where $[\chi]=[\chi_P]$ in $\HL^2(\Omega\, u(\mathfrak h))$.
\end{lem}
\begin{proof}
Let $[(\Theta,\chi)]$ be an equivalence class in $\mathcal H^2(\T)$ given by some PD datum $\mathscr D=(\T,\Theta,\chi)$. By Proposition \ref{P:RCdata} (i), $\chi$ is a $2$-cocycle but not a $2$-coboundary in $\Omega^2\,(u(\mathfrak h))$. Hence, there exists a point $P\in \mathbb A_\field^{n(n+1)/2}\setminus\{0\}$ such that $[\chi_P]=[\chi]$ in $\HL^2(\Omega\,u(\mathfrak h))$ by Remark \ref{R:H2A3}. Next, we write $\chi_P=\chi+\partial^1(s)$ for some $s\in \mathfrak u(\mathfrak h)^+$. It is easy to check that $\mathscr D'=(\T,\Theta+\Phi_z(s),\chi_P)$ is also a PD datum and $(\Theta,\chi)\sim (\Theta+\Phi_z(s),\chi_P)$. So without loss of generality, we can assume that $[(\Theta,\chi)]=[(\Theta,\chi_P)]$ for some $P\in \mathbb A_\field^{n(n+1)/2}\setminus\{0\}$. 

Now, suppose $(\Theta_P,\chi_P)\sim (\Theta_Q,\chi_Q)$ for two points $P,Q\in \mathbb A^{n(n+1)/2}\setminus \{0\}$. By Definition \ref{D:EquiR}, there exists some $s\in u(\mathfrak h)^+$ such that $\chi_P-\chi_Q=\partial^1(s)$ and $\Theta_P-\Theta_Q=\Phi_z(s)$. It follows that $[\chi_P]=[\chi_Q]$ in $\HL^2(\Omega\, u(\mathfrak h))$, which implies that $P=Q$ by Remark \ref{R:H2A3} again. Hence, for any $[(\Theta,\chi_P)]$ in $\mathcal H^2(\T)$, the point $P$ is uniquely determined by the equivalent class. And there is a well-defined map $q: \mathcal H^2(\T)\to \mathbb A_\field^{n(n+1)/2}\setminus\{0\}$ induced by $[(\Theta,\chi_P)]\mapsto P$. 

By Remark \ref{R:APCD}, $q$ maps $\mathcal H^2(\T)$ onto $\mathscr A^+(\T)$. It remains to show that $q$ is injective. Suppose there are two equivalence classes $[(\Theta,\chi_P)]$ and $[(\Theta',\chi_P)]$ which are in $q^{-1}(P)$. By Proposition \ref{P:RCdata} (iii), $\Phi_z(\chi_P)=\partial^1(\Theta)=\partial^1(\Theta')$. Hence, $\partial^1(\Theta-\Theta')=0$. Then, there is some $s\in \mathfrak h$ such that $\Theta=\Theta'+s$ since $\Ker\,\partial^1=\HL^1(\Omega\, u(\mathfrak h))=\mathfrak h$ by Lemma \ref{L:HochUH}. On the other hand, by Proposition \ref{P:RCdata} (ii), $\rho_z(s)=\rho_z(\Theta)-\rho_z(\Theta')=0$. Since $\T$ is permissible, there is some $r\in \mathfrak h$ such that $s=\Phi_z(r)$. Then one sees that $\chi_P-\chi_P=0=\partial^1(r)$, since $r \in \mathfrak h$ is primitive, and $\Theta-\Theta'=s=\Phi_z(r)$. This implies that $(\Theta,\chi_P)\sim (\Theta',\chi_P)$ in $\mathcal H^2(\T)$, and hence $q$ is injective. %This completes the proof. 
\end{proof}

\begin{thm}\label{T:BPT}
Let $\T$ be a permissible abelian type of rank $n$. There is a bijection between $\mathcal H^2(\T)/\Aut(\T)$ and $\mathscr A^+(\T)/\Aut(\T)$ induced by $\xymatrix{
\left[(\Theta,\chi_P)\right]\ar@{<->}[r]&[P],
}
$
where
\[
P=(a_{ij},b_k)_{1\le i<j\le n,1\le k\le n}\in \mathscr A^+(\T),\ \chi_P=\sum_{1\le i<j\le n}a_{ij}\, x_i\otimes x_j+\Bock\left(\sum_{1\le k\le n}b_k\, x_k\right)\]
and $\Theta$ is chosen from $u(\mathfrak h)^+$ (not necessarily unique) such that $\rho_z(\Theta)=0$ and $\Phi_z(\chi_P)=\partial^1(\Theta)$.
\end{thm}
\begin{proof}
By Lemma \ref{L:BIH2A2}, it suffices to prove that the $\Aut(\T)$-action on $\mathscr A^+(\T)$ induced by the bijection $q: \mathcal H^2(\T)\longleftrightarrow\mathscr A^+(\T)$ is compatible with Lemma \ref{L:AutAction}, which follows from Lemma \ref{L:PCDAct} (ii) and Lemma \ref{L:ActionAT}.
\end{proof}

\begin{remark}
By Theorem~\ref{TH:HopfO}, there is a bijection between the set of isomorphism classes of quantum $p$-groups in $\mathscr X(n)$ and the disjoint union
\[
\sqcup_{\T\in \mathcal T_P} \mathscr A^+(\T)/\Aut(\T)\, \bigsqcup\, \sqcup_{\T\in \mathcal T_N}\mathcal H^2(\T)/\Aut(\T),
\]
where $\mathcal T_P$ is the set of all non-isomorphic permissible types and $\mathcal T_N$ is the set of all non-isomorphic nonpermissible types. Thus, to apply Primitive Deformation theory in the next section for $\mathscr X(2)$, our approach is to break the classification into permissible and nonpermissible types.  
\end{remark}

%%%%%%%%%%%%%%%%%%%%%%%%%%%%%%%%

\section{Primitive deformations of rank-two abelian types} 
\label{sec:X(2)}

Again, set characteristic $p>2$ throughout this section. We apply our Primitive Deformation theory to all abelian types of rank $n=2$. The results in this section provide the classification of connected $p^3$-dimensional Hopf algebras over $\field$ with abelian primitive space of dimension $2$, i.e., classification of quantum $p$-groups in $\mathscr X(2)$. Throughout, let $\T=(\mathfrak g,\mathfrak h,\rho)$ be an abelian type of rank $2$, where we fix bases $z$ of $\mathfrak g$ and $x,y$ of $\mathfrak h$. Regarding Remark \ref{R:H2A3}, for any $P=(a,b,c)\in \mathbb A_\field^3$, we denote $\chi_P=a\,x\otimes y+\Bock (b\,x+c\,y)$. According to Lemma \ref{L:AutAction}, the automorphism group $\Aut(\T)$ acts on $\mathbb A_\field^3$ via, for any $\phi\in \Aut(\T)$,
\begin{align}\label{E:EmA3}
\phi=(\gamma,G)\mapsto \left(
\begin{array}{ccc}
\gamma\, \text{det}(G) &  0  & 0\\
0  &  \multicolumn{2}{c}{\multirow{2}{*}{{\large $\gamma^{\frac{1}{p}}\, G^{\top}$}}}\\
0  &
\end{array}
\right).
\end{align}

%%%%%%%%%%%%%%%%%%%%%%%%%%%%%%%%

\subsection{Types}
First of all, we classify all types $\T=(\mathfrak g,\mathfrak h,\rho)$ of rank $2$. Note that $\mathfrak g$ and $\mathfrak h$ are abelian Lie algebras, i.e., equipped with trivial Lie bracket. These two restricted Lie algebras are determined by their restricted maps on the bases. After taking some linear transformations, there are only two isomorphism classes of $\mathfrak g$ and four isomorphism classes of $\mathfrak h$ \cite[Theorem~7.4 (1)-(4)]{wang2012connected}. The isomorphism classes are labeled below and distinguished by the restricted maps
\begin{align*}
\mathfrak g: \ & (N)\ z^{[p]}=0; &&(S)\ z^{[p]}=z;\\
\mathfrak h: \ & (A)\ x^{[p]}=0, y^{[p]}=0; && (B)\ x^{[p]}=x, y^{[p]}=0; \\  
                    \ & (C)\ x^{[p]}=y, y^{[p]}=0; && (D)\ x^{[p]}=x, y^{[p]}=y.
\end{align*}
Now, any fixed pair $(\mathfrak g,\mathfrak h)$ is, up to isomorphism, one of the possible combinations of the $\mathfrak g$'s and $\mathfrak h$'s from the above list. The restricted map of $\mathfrak g$ can only be $z^{[p]}=\delta z$, where $\delta=0,1$. %For the remaining of this paper, we will keep the notation $\delta=0,1$. 
We use $e_{ij}$ to denote the $2\times 2$ matrix with 1 in the $(i,j)$-entry and zero elsewhere.

%combine original Prop. 5.1 and Lemma 5.2 together
We can always assume that two isomorphic types share the same restricted Lie algebras by changing bases and rewriting the corresponding algebraic representation using the commutative diagram in Definition \ref{D:2TIsom}. Hence, applying Lemma \ref{L:IsomCT}, one can directly compute $\T$ for each pair of isomorphism classes of $\mathfrak g$ and $\mathfrak h$. 

In summary, there are totally 14 isomorphism classes and one parametric family of abelian types of rank $2$, which are listed in Table~\ref{tab:rank 2}. We only illustrate the calculation for the parametric family $\T(\zeta)$, where $\mathfrak g=S,\mathfrak h=A$, the parameter $\zeta\in \mathbb F_p$, and nonzero $\zeta$ is determined up to replacement of $\zeta^{-1}$. We will leave the rest to the reader. \\

\noindent
\textbf{Type $\T(\zeta)$:} Using the same notations as in Section~\ref{sec:realization}, first of all, we find all possible algebraic representations of $\mathfrak g$ on $\mathfrak h$. By Lemma \ref{L:IsomCT}, they are all given by $2\times 2$ matrices $M$ satisfying $M^p=M$, since $R=0$ for $\mathfrak h=A$. Next, for $\Aut(\mathfrak h)$, we can choose $G$ to be any invertible $2\times 2$ matrix. Similarly for $\mathfrak g$, one sees that $\gamma$ should be a nonzero scalar in the field $\mathbb F_p$ of exactly $p$ elements. In conclusion, the set of isomorphism classes of types with fixed pair $(S,A)$ is bijective to the set $\{M\in \text{M}_2(\field)\,|\,M^p=M\}$ modulo the equivalence relation $\sim$, where $M\sim M'$ if $M'=\gamma G^{-1}MG$, for some $G\in \GL_2(\k)$ and $\gamma\in \mathbb F_p^\times$. 

Without loss of generality, suppose $M\neq 0$, otherwise it is (\T3). One sees that $M$ is diagonalizable with eigenvalues inside $\mathbb F_p$ since $M^p=M$. Hence, by a linear transformation $G$ of the base and setting $\gamma=1$, we can always assume that $M$ is a diagonal matrix with eigenvalues inside $\mathbb F_p$. By further rescaling one nonzero eigenvalue by $\gamma^{-1}$, we have $M\sim \Diag(1,\zeta)$ for some $\zeta\in \mathbb F_p$. Applying the equivalence relation $\sim$ again, one verifies that the equivalence classes of $\Diag(1,\zeta)$ with $\zeta \in \mathbb F_p\setminus\{0\}$ are determined up to replacement of $\zeta^{-1}$. Hence, one has $M=e_{11}+\zeta e_{22}$ and $\T(\zeta)$ is permissible by a direct computation. %Finally, in order to show that $\T(\zeta)$ is permissible, one can check that $\Img\, \Phi_z=\Ker\, \rho_z = 0$.

\FloatBarrier
\begin{table}[!htp] % table is a floating environment. TeX has special rules for placing floats so that they cause minimal disruption of the surrounding text. You can give LaTeX hints of how to place your floats. For example, to make LaTeX try to place it where you inserted it, specify the option h (for ?here?). Note that this won?t always work because LaTeX still tries to make the text look nice. Your hint is only that: a hint, not a definite command. You can force LateX to obey your command to place a floating environment HERE by loading the package "float" and specifying the option H (capital letter!). But notice that this can lead to an ugly layout.
\caption {Classify types $\T=(\mathfrak g, \mathfrak h, \rho_z)$ of rank $2$} \label{tab:rank 2} \vspace{-0.2cm}
\begin{center}
\begin{tabular}{| c | c | c | c | c | c |}  %{|l|l|l|l|}  %
\hline
\bf{Types}  & $\mathfrak{g}$ & $\mathfrak{h}$ & $\rho_z=M$ & \bf{Permissible} & $\mathscr A^+(\T)$ \\
\hline \hline
(\T1) & $N$ & $A$ & $0$ & No, $\Img\, \Phi_z=0, \ \Ker\, \rho_z = \mathfrak h$ & Non-empty \\ \hline
(\T2) & $N$ & $A$ & $e_{12}$ & No, $\Img\, \Phi_z=0, \ \Ker\, \rho_z = \k\, y$ & Non-empty \\ \hline
(\T3) & $S$ & $A$ & $0$ & Yes, $\Img\, \Phi_z=\Ker\, \rho_z = \mathfrak h$ & Empty \\ \hline
(\T4) & $N$ & $B$ & $0$ & No, $\Img\, \Phi_z=\k\, x, \ \Ker\, \rho_z = \mathfrak h$ & Non-empty \\ \hline
(\T5) & $N$ & $B$ & $e_{21}$ & Yes, $\Img\, \Phi_z=\Ker\, \rho_z = \k\, x$ & Non-empty \\ \hline
(\T6) & $S$ & $B$ & $0$ & Yes, $ \Img\, \Phi_z=\Ker\, \rho_z = \mathfrak h$ & Non-empty \\ \hline
(\T7) & $S$ & $B$ & $e_{22}$ & Yes, $\Img\, \Phi_z=\Ker\, \rho_z = \k\, x$ & Non-empty \\ \hline
(\T8) & $S$ & $B$ & $e_{21}+e_{22}$ & Yes, $\Img\, \Phi_z=\Ker\, \rho_z = \k\, x$ & Non-empty \\ \hline
(\T9) & $N$ & $C$ & $0$ & No, $\Img\, \Phi_z= \k\, y, \ \Ker\, \rho_z = \mathfrak h$ & Non-empty \\ \hline
(\T10) & $N$ & $C$ & $e_{12}$ & Yes, $\Img\, \Phi_z=\Ker\, \rho_z = \k\, y$ & Non-empty \\ \hline
(\T11) & $S$ & $C$ & $0$ & Yes, $\Img\, \Phi_z=\Ker\, \rho_z = \mathfrak h$ & Empty \\ \hline
(\T12) & $S$ & $C$ & $e_{11}$ & Yes, $\Img\, \Phi_z=\Ker\, \rho_z = \k\, y$ & Non-empty \\ \hline
(\T13) & $N$ & $D$ & $0$ & Yes, $\Img\, \Phi_z=\Ker\, \rho_z = \mathfrak h$ & Empty, \\ \hline
(\T14) & $S$ & $D$ & $0$ & Yes, $\Img\, \Phi_z=\Ker\, \rho_z = \mathfrak h$ & Non-empty \\ \hline
$\T(\zeta=-1), \zeta\in \mathbb F_p$ & $S$ & $A$ & $e_{11}+\zeta e_{22}$ & Yes, $\Img\, \Phi_z=\Ker\, \rho_z = 0$ & Empty \\ \hline
$\T(\zeta \neq -1), \zeta\in \mathbb F_p$ & $S$ & $A$ & $e_{11}+\zeta e_{22}$ & Yes, $\Img\, \Phi_z=\Ker\, \rho_z = 0$ & Non-empty \\ \hline
\end{tabular}
\end{center}
\end{table}

Next, we will show that set $\mathscr A^+(\T)$ is empty for the case $\T(\zeta=-1)$, and leave the rest to the reader. According to Table~\ref{tab:rank 2}, we have $z^{[p]}=z$ and $\mathfrak h^{[p]}=0$ and $\rho_z(x)=x,\, \rho_z(y)=-y$.  Let's assume $P=(a,b,c)\in \mathscr A^+(\T)$. By Remark \ref{R:APCD}, there is some PD datum $(\Theta,\chi_P)$ with respect to $z$. We use the fact that $\rho_z\left(\Bock(\mathfrak h)\right)=0$ in $\HL^2(\Omega\,u(\mathfrak h))$ by Proposition~\ref{P:Cobar} (iii). Hence
\begin{align*}
\Phi_z(\chi_P)&\,=[\chi_P^p- \chi_P+\rho_z^{p-1}(a\,x\otimes y)]\\
&\,=[-\chi_P+a\,\rho_z^{p-2}(\rho_z(x)\otimes y+x\otimes \rho_z(y))]=[-\chi_P]=0.
\end{align*}
So $\chi_P$ is a $2$-coboundary, which is a contradiction according to Proposition \ref{P:RCdata} (i). Therefore, $\mathscr A^+(\T)$ is empty for $\T(\zeta=-1)$.

\begin{remark} \label{R:primC}
When $\T$ is one of those types in Table~\ref{tab:rank 2} whose set $\mathscr A^+(\T)$ is empty, there is no PD of $u(\T)$. Indeed, $u(\T)$ for these types are isomorphic to certain connected Hopf algebras in the C-family classified in \cite[Theorem~1.4]{NWW1}, also see Table~\ref{tab:typeC}. In particular, $u(\T3)\cong (\C7)$, $u(\T11)\cong (\C8)$ and $u(\T13)\cong (\C9)$ and $u(\T(\zeta=-1))\cong (\C16)$ with parameters $\lambda=-i$ and $\delta=(-i)^{p-1}$, where $i^2=-1$.
\end{remark}

\subsection{Primitive deformations of permissible types}\label{S:typesT}

Let $\T$ be a permissible type in Table~\ref{tab:rank 2} with non-empty set $\mathscr A^+(\T)$, $[P]$ be one of the $\Aut(\T)$-orbits in $\mathscr A^+(\T)$. By Theorem~\ref{TH:HopfO}, isomorphism classes of quantum $p$-groups in $\mathscr X(2)$ of type $\T$ can be obtained by PD's of $u(\T)$ with the PD data taken from the equivalence classes $[(\Theta,\chi_P)]$ in $\mathscr A^+(\T)/\Aut(\T)$. If an orbit $[P]$ is a single point, then there is only one isomorphism class of quantum $p$-groups given by $[P]$. If an orbit $[P]$ is infinite, e.g., $[P]=[(\xi,0,1)]$ for the type (\T5) with $\xi\in \field$, the quantum $p$-groups given by values $\xi,\xi'\in\field$ are isomorphic if and only if $\xi=\tau\xi'$ for some $n$-th root of unity $\tau$. Then we say that the isomorphism classes of the quantum $p$-groups given by the orbit $[P]$ are parametrized by $\field/\mu_n$ (or, $\field^\times/\mu_n$).

Below in Table \ref{tab:ConC}, for each permissible type $\T$ in Table~\ref{tab:rank 2} with non-empty set $\mathscr A^+(\T)$, we give the $\Aut(\T)$-orbits $[P]$ in $\mathscr A^+(\T)$, the equivalence classes $[(\Theta,\chi_P)]$ in $\mathscr A^+(\T)/\Aut(\T)$, and the isomorphism classes of the quantum $p$-groups described by the parametric spaces, respectively. In Appendix~\ref{Appen:ConC}, we check for the type (\T5) and leave the rest to the reader.

\FloatBarrier
\begin{table}[!htp] % table is a floating environment. TeX has special rules for placing floats so that they cause minimal disruption of the surrounding text. You can give LaTeX hints of how to place your floats. For example, to make LaTeX try to place it where you inserted it, specify the option h (for ?here?). Note that this won?t always work because LaTeX still tries to make the text look nice. Your hint is only that: a hint, not a definite command. You can force LateX to obey your command to place a floating environment HERE by loading the package "float" and specifying the option H (capital letter!). But notice that this can lead to an ugly layout.
\caption{Quantum $p$-groups of permissible types}\label{tab:ConC} \vspace{-0.2cm}
\begin{center}
\begin{tabular} {| c | p{2.4cm} | p{6.9cm} | p{2.1cm} |} %{|l|l|l|l|}
\hline
\bf{Types}  &  $[P]$   & $[(\Theta,\chi_P)]$ & \bf{\# of Iso. classes} \\
\hline \hline
(T5)       &   $P=(1,0,0)$                                                      &  $\Theta=0,\, \chi_P=x\otimes y$                                                          &              1               \\
%           &    $(0,0,1)$                                                      &     $\Theta=(x^{p-1}-1)y, \chi=\Bock(y)$                                                         &                            \\
           &    $P=(\xi,0,1)$                                                      &     $\Theta=(x^{p-1}-1)y,\, \chi_P=\xi x\otimes y+\Bock(y)$                                                         &    $\k/\mu_{(p^2-1)/2}$                         \\
\hline
(T6)       &    $P=(1,0,0)$                                                      &   $\Theta=0,\, \chi_P=\Bock(x)$                                                           &           1                  \\
\hline
(T7)       &    $P=(1,0,0)$                                                      &        $\Theta=0,\, \chi_P=x\otimes y$                                                     &         1                    \\
       &     $P=(0,1,0)$                                                     &     $\Theta=0,\, \chi_P=\Bock(x)$                                                          &                  1           \\
              &   $P=(1,1,0)$                                                       &       $\Theta=0,\, \chi_P=x\otimes y+\Bock(x)$                                                        &         1                    \\
\hline
%$T8$          &     $(0,1,0)$                                                     & $\Theta=0,\chi=\Bock(x)$                                                             &                            \\
(T8)       &     $P=(\xi,0,0)$                                                     & $\Theta=-\frac{\xi}{2}x^2,\, \chi_P=\xi x\otimes y$                                                             &            $\k^\times /\mu_{(p-1)/2}$                 \\

          &     $P=(\xi,1,0)$                                                     & $\Theta=-\frac{\xi}{2}x^2,\, \chi_P=\xi x\otimes y+\Bock(x)$                                                             &      $\k $                       \\
\hline
(T10)       &   $P=(1,0,0)$                                                       &      $\Theta=0,\,\chi_P=x\otimes y$                                                         &              1               \\
%       &   $(0,0,1)$                                                       &      $\Theta=0,\chi=\Bock(y)$                                                         &                             \\
       &   $P=(\xi,0,1)$                                                       &      $\Theta=0,\, \chi_P=\xi x\otimes y+\Bock(y)$                                                         &    $\k /\mu_{p^2-p+1}$                          \\
\hline
(T12)       &     $P=(1,0,0)$                                                     &        $\Theta=0,\, \chi_P=x\otimes y$                                                        &           1                  \\
\hline
(T14)      &         $P=(1,0,0)$                                                 &         $\Theta=0,\, \chi_P=x\otimes y$                                                        &         1                    \\
      &         $P=(0,1,0)$                                                 &         $\Theta=0,\, \chi_P=\Bock(x)$                                                        &                    1         \\
      &         $P=(1,1,0)$                                                 &         $\Theta=0,\, \chi_P=x\otimes y+\Bock(x)$                                                        &    1                         \\
\hline
T$(\zeta\neq -1)$      &    $P=(1,0,0)$                                                      &        $\Theta=0,\, \chi_P=x\otimes y$                                                     &         $(p+1)/2$                \\
\hline
\end{tabular}
\end{center}
\end{table}

Note that these types contain four infinite families, which are all parametrized by one parameter denoted by $\xi$. For the type $\T(\zeta\neq -1)$, the isomorphism classes form a finite parametric family with parameter $\zeta\in \mathbb F_p\setminus\{-1\}$. By Table~\ref{tab:rank 2}, $\T(\zeta)\cong \T(\zeta')$ if and only if $\zeta=\zeta'^{\pm 1}$ when $\zeta,\zeta'\neq 0$. Hence, there are total $\frac{p+1}{2}$ isomorphism classes for $\T(\zeta\neq -1)$. For each of the other cases in Table \ref{tab:ConC}, the isomorphism class is unique.\\

%%%%%%%%%%%%%%%%%%%%%%%%%%%%%%%%

\subsection{Primitive deformations of nonpermissible types}\label{typesnotC}
In this subsection, we will deal with nonpermissible abelian types of rank $2$, i.e., those from the set 
\[\mathcal E:=\{(\T1),(\T2),(\T4),(\T9)\}\]
in Table~\ref{tab:rank 2}. In the following, let $\T=(\mathfrak g,\mathfrak h,\rho)\in \mathcal E$, where we still fix bases $x,y$ for $\mathfrak h$, $z$ for $\mathfrak g$ and write $z^p=\delta z$ for $\delta=0,1$. We keep using the same notations and definitions as before. The restricted universal enveloping algebra $u(\mathfrak h)$ has a PBW basis $\{x^iy^j\, |\, 0\le i,j\le p-1\}$, where we denote by $u(\mathfrak h)_{\ge 2}$ the subspace of $u(\mathfrak h)$ spanned by all those PBW bases $x^iy^j$ satisfying $i+j\ge 2$. Clearly, there is a vector space decomposition $u(\mathfrak h)^+=u(\mathfrak h)_{\ge 2}\oplus \mathfrak h$.

First, we will show that there is an one-to-one correspondence between the equivalence classes in $\mathcal H^2(\T)$ and the points in a subset of the affine space $\mathbb A^5$, which will be constructed as follows.

Let $P=(a,b,c,d,e)\in \mathbb A^5=\mathbb A^2\times \mathbb A^3$. We denote a pair $(\Theta_P,\chi_P)$, where $\Theta_P\in \mathfrak h$ and $\chi_P\in (u(\mathfrak h)^+)^{\otimes 2}$ such that
\begin{align}\label{E:ST}
\Theta_P=a\, x+b\, y,\quad \chi_P=c\, x\otimes y+\Bock(d\, x+e\,y).
\end{align}

\begin{deff}\label{D:BSet}
Let $P=(a,b,c,d,e)\in \mathbb A^5=\mathbb A^2\times \mathbb A^3$, $\Theta_P$ and $\chi_P$ as in Equation~\eqref{E:ST}. Suppose $\T\in \mathcal E$. We define 
\[
B^+(\T)=\set[\Bigg]
{
P\in \mathbb A^5\given
\begin{tabular}{l }
$\chi_P\neq 0$,\, \text{and there exists some}\ $\Theta\in u(\mathfrak h)_{\ge 2}$\\ \text{such that}\, $\Phi_z(\chi_P)=\partial^1(\Theta),\, \rho_z(\Theta+\Theta_P)=0$
\end{tabular}
}
\]
and
\[
\mathscr B^+(\T)=
\begin{cases}
B^+(\T), & \T=(\T1)\, \text{or}\, (\T2)\\
\text{The set of}\, P\in B^+(\T)\, \text{with}\, a=0,  &\T=(\T4)\\ 
\text{The set of}\, P\in B^+(\T)\, \text{with}\, b=0,   &\T=(\T9) 
\end{cases}.
\]
\end{deff}

\begin{remark}\label{R:A5}
\
\begin{enumerate}
\item[(i)] Let $P'=(c,d,e)$ be the projection of $P=(a,b,c,d,e)$ in $\mathbb A^3$. Indeed, $\chi_P$ is the same as the element $\chi_{P'}$ defined in Remark \ref{R:H2A3}.
\item[(ii)] Note that the element $\Theta$ used to describe the set $B^+(\T)$ in Definition \ref{D:BSet} is uniquely determined by $P$ since two candidates $\Theta$ and $\Theta'$ yields $\partial^1(\Theta-\Theta')=0$, which implies that $\Theta-\Theta'\in \left(\,\mathfrak h \cap u(\mathfrak h)_{\geq 2}\,\right) = 0$. We then denote such $\Theta$ by $\Psi_P$ for any $P\in B^+(\T)$. 
\item[(iii)] By definition, a point $P\in \mathbb A^5$ belongs to $\mathscr B^+(\T)$ if and only if there is a PD datum $(\T,\Psi_P+\Theta_P,\chi_P)$ with respect to $z$, for some $\Psi_P\in u(\mathfrak h)_{\ge 2}$.
\end{enumerate}
\end{remark}

\begin{deff}\label{D:equiv}
We define an equivalence relation $\sim$ on $B^+(\T)$ such that two points $P\sim Q$ in $B^+(\T)$ if and only if 
\begin{align*}%\label{E:Equiv}
\Theta_P-\Theta_Q=\Phi_z(s),\ \text{and}\ \chi_P=\chi_Q,
\end{align*}
for some $s\in \mathfrak h$. 
\end{deff}

\begin{lem}\label{L:BIJ}
Let $\T\in \mathcal E$. Then the following are bijective 
\begin{itemize}
\item[(i)] $\mathcal H^2(\T)$;
\item[(ii)] $B^+(\T)/\sim$;
\item[(iii)] $\mathscr B^+(\T)$,
\end{itemize}
where the bijection between (i) and (ii) is induced by $[(\Psi_P+\Theta_P,\chi_P)]\longleftrightarrow [P]$ and the bijection between (ii) and (iii) is given by the natural projection from $B^+(\T)$ to $\mathscr B^+(\T)$.
\end{lem}
\begin{proof}
It follows by the same argument in Lemma \ref{L:BIH2A2} that there is a bijection between $\mathcal H^2(\T)$ and $B^+(\T)/\sim$ induced by $[(\Psi_P+\Theta_P,\chi_P)]\longleftrightarrow [P]$. 

It remains to show that there is a bijection between $B^+(\T)/\sim$ and $\mathscr B^+(\T)$ via $[P]\longleftrightarrow Q$ where $Q$ is the projection of $P$. For type (\T1) or (\T2), the sets $B^+(\T)=\mathscr B^+(\T)$ are the same. So it suffices to show that $\sim$ is trivial. It is easy to check that $\Phi_z(\mathfrak h)=0$. Then, by Definition~\ref{D:equiv}, $P\sim Q$ in $B^+(\T)$ if and only if $P=Q$. For type (\T4), denote by $Q=(0,b,c,d,e)$ the projection of $P=(a,b,c,d,e)\in B^+(\T)$. We claim that $Q\in \mathscr B^+(\T)$. Since $\Phi_z(\mathfrak h)=\field\, x$, there is some $s\in \mathfrak h$ such that $ax=\Phi_z(s)$. By Remark~\ref{R:A5} (ii), we set $\Psi_Q=\Psi_P\in u(\mathfrak h)_{\ge 2}$. By definition, $\chi_Q=\chi_P\neq 0$ and $\Phi_z(\chi_Q)=\Phi_z(\chi_P)=\partial^1(\Psi_P)=\partial^1(\Psi_Q)$. By Proposition \ref{P:Cobar} (ii),
\begin{align*}
\rho_z(\Psi_Q+\Theta_Q)& =\rho_z (\Psi_P + by) = \rho_z(\Psi_P+a\,x+b\,y)-\rho_z(a\,x) \\
&=\rho_z(\Psi_P+\Theta_P)-(\rho_z\circ \Phi_z)(s)=0.
\end{align*}
By Remark~\ref{R:A5} (iii), $Q\in \mathscr B^+(\T)$. Moreover, $P\sim Q$ since $\Theta_P - \Theta_Q = \Phi_z(s)$ and $\chi_P=\chi_Q$ by Definition~\ref{D:equiv}. Now every point in $B^+(\T)$ is equivalent to its projection in $\mathscr B^+(\T)$ where $\sim$ is trivial, so the bijection between $B^+(\T)/\sim$ and $\mathscr B^+(\T)$ is established. The argument for $(\T9)$ is similar. This completes the proof. 
\end{proof}

In Table \ref{tab:SetWC}, we give a description of the sets $\mathscr B^+(\T)$ for types $\T\in \mathcal E$. For a point $P\in \mathscr B^+(\T)$ with coordinates $a,b,c,d$ and $e$, we also list the corresponding elements $\Psi_P,\Theta_P$ and $\chi_P$ that appeared in Definition \ref{D:BSet} and Remark~\ref{R:A5} (ii). Moreover, the corresponding equivalence class in $\mathcal H^2(\T)$ is the pair $[(\Psi_P+\Theta_P,\chi_P)]$, which gives a PD datum $\mathscr D=(\T,\Psi_P+\Theta_P,\chi_P)$ with respect to $z$.  In Appendix~\ref{Appen:SetWC}, we check for the type (T2) and leave the rest to the reader. 

\FloatBarrier
\begin{table}[!htp] % table is a floating environment. TeX has special rules for placing floats so that they cause minimal disruption of the surrounding text. You can give LaTeX hints of how to place your floats. For example, to make LaTeX try to place it where you inserted it, specify the option h (for ?here?). Note that this won?t always work because LaTeX still tries to make the text look nice. Your hint is only that: a hint, not a definite command. You can force LateX to obey your command to place a floating environment HERE by loading the package "float" and specifying the option H (capital letter!). But notice that this can lead to an ugly layout.
\caption{$\mathscr B^+(\T)$ for nonpermissible types} \label{tab:SetWC} \vspace{-0.2cm}
\begin{center}
\begin{tabular}{| c | p{5.5cm} | c | c | p{3.8cm} |} %{| l| l | l  | l  |l|}
\hline 
\bf{Types} &  $\mathscr B^+(\T)$  &  $\Psi_P$  & $\Theta_P$ & $\chi_P$\\
\hline \hline
(\T1)  & $\{(a,b,c,d,e)\}\setminus\{(a,b,0,0,0)\}$ & $0$  & $a\, x+b\, y$ &   $c\, x\otimes y+\Bock(d\, x+e\, y)$\\
\hline
(\T2)  & $\{(0,b,c,d,e)\}\setminus \{(0,b,0,0,0)\}$ & $d^p\, xy^{p-1}$  & $b\, y$ &   $c\, x\otimes y+\Bock(d\, x+e\, y)$\\
\hline
(\T4)  & $\{(0,b,c,0,e)\}\setminus \{(0,b,0,0,0)\}$ & $0$  & $b\,y$ &   $c\,x\otimes y+\Bock(e\, y)$\\
\hline
(\T9)  & $\{(a,0,c,0,e)\}\setminus\{(a,0,0,0,0)\}$ & $0$  & $a\,x$ &   $c\,x\otimes y+\Bock(e\, y)$\\
\hline
\end{tabular}
\end{center}
\end{table}

Next, we establish an $\Aut(\T)$-action on $\mathscr B^+(\T)$, whose orbits are bijective to $\mathcal H^2(\T)/\Aut(\T)$. Since $\mathscr B^+(\T)\in \mathbb A^5$ is bijective to $\mathcal H^2(\T)$ by Lemma \ref{L:BIJ}, it suffices to find the induced $\Aut(\T)$-action on $\mathscr B^+(\T)$ according to Lemma \ref{L:PCDAct}. The following lemma comes directly from conditions (\ref{E:RepC}).

\begin{lem}\label{L:AutNP}
Let $\T\in \mathcal E$, and $\phi=(\gamma,G)\in \Aut(\T)$. Then $\gamma\neq 0$ and the entries of $G=(g_{ij})_{2\times 2}$ satisfy the following 
\[
\begin{cases}
g_{11}g_{22}-g_{12}g_{21}\neq 0 & \T=(\T1)\\
g_{21}=0,\, g_{11}=\gamma g_{22}\neq 0 &\T=(\T2)\\
g_{21}=g_{12}=0,\, g_{11},g_{22}\neq 0  &  \T=(\T4)\\
g_{21}=0,\, g_{22}=g_{11}^p,\, g_{11}\neq 0  &\T=(\T9).
\end{cases} 
\]
\end{lem}

Similar to the embedding \eqref{E:EmA3}, there is an embedding $\Aut(\T)\hookrightarrow \text{GL}_5(\field)$, which implies an $\Aut(\T)$-action on $\mathbb A^5_\field$ by left multiplication on columns, that is
\begin{align}\label{E:ActionG2}
\phi(P)
=&\ \left(
\begin{array}{ccccc}
\multicolumn{2}{c}{\multirow{2}{*}{{\large $\gamma^p\, G^{\top}$}}}& 0 &0 &0\\
& & 0&0&0\\
0&0 &\gamma\ \text{det}(G)  &  0& 0\\
0&0&0& \multicolumn{2}{c}{\multirow{2}{*}{{\large $\gamma^{\frac{1}{p}}\, G^{\top}$}}}\\
0 &0 &0& 
\end{array}
\right)
\begin{pmatrix}
a\\ b\\ c\\d\\e
\end{pmatrix}\\
=&\ \left(\gamma^{p}\,(g_{11}a+g_{21}b),\, \gamma^{p}\,(g_{12}a+g_{22}b),\, \gamma\,(g_{11}g_{22}-g_{12}g_{21})c,\right.\notag\\
  &\quad\quad\quad\quad\quad\left. \gamma^{1/p}\,(g_{11}d+g_{21}e),\, \gamma^{1/p}\,(g_{12}d+g_{22}e)\right),&\notag
\end{align}
for any $P=(a,b,c,d,e)\in \mathbb A_\field^5$ and $\phi=(\gamma,G)\in \Aut(\T)$ with $G=(g_{ij})_{2\times 2}$. 

In case $\T=(\T9)$, in order to have a well-defined group-action on $\mathscr B^+(\T)$, we replace $\Aut(\T)$ by the subgroup
\[K=\{\phi=(\gamma,G)\,|\,\gamma\neq 0, g_{12}=g_{21}=0, g_{22}=g_{11}^p\, \text{and}\, g_{11}\neq 0\}.\] 
Note that $K$ is also isomorphic to a quotient group of $\Aut(\T)$ by the subgroup $\{\phi=(\gamma,G)\in \Aut(\T)\, |\, \gamma=1, g_{11}=g_{22}=1\}$. For the sake of notation, we set
\[
\widetilde{\Aut}(\T)=
\begin{cases}
\Aut(\T)   &  \T=(\T1), \,(\T2),\, \text{or}\, (\T4)\\
K & \T=(\T9).
\end{cases}
\]

\begin{thm}\label{T:NPerT}
Let $\T\in \mathcal E$. There is a bijection $\mathcal H^2(\T)/\Aut(\T)\longleftrightarrow \mathscr B^+(\T)/\widetilde{\Aut}(\T)$ induced by
\[
\xymatrix{
\left[(\Psi_P+\Theta_P,\chi_P)\right]\ar@{<->}[rr]&&[P],
}
\]
where $P=(a,b,c,d,e)\in \mathscr B^+(\T)$, $\Theta_P=a\,x+b\,y$, $\chi_P=c\, x\otimes y+\Bock(d\, x+e\, y)$ and $\Psi_P\in u(\mathfrak h)_{\ge 2}$ given in Definition \ref{D:BSet} and Remark~\ref{R:A5} (ii).
\end{thm}
\begin{proof}
By Lemma \ref{L:BIJ}, $\mathcal H^2(\T)$ and $\mathscr B^+(\T)$ are naturally bijective to each other. Hence, it suffices to find the corresponding $\Aut(\T)$-action on $\mathscr B^+(\T)$ and match it with Equation \eqref{E:ActionG2}. We will check for the case $\T=(\T9)$ and leave the rest to the reader. 

By Table \ref{tab:SetWC}, we choose any point $P=(a,0,c,0,e)\in \mathscr B^+(\T)$, where $c,e$ are not both zero. Since $\Psi_P=0$, the corresponding equivalence class in $\mathcal H^2(\T)$ is given by $[(\Theta_P,\chi_P)]$ with $\Theta_P=a\, x$ and $\chi_P=c\, x\otimes y+\Bock(e\, y)$. Let $\phi=(\gamma,\phi_2)=(\gamma,G)\in \Aut(\T)$. By Lemma \ref{L:AutNP}, we can write $G=\begin{pmatrix} g_{11} & g_{12}\\ 0 & g_{22}\end{pmatrix}$ with $g_{22}=g_{11}^p\neq 0$. 

We will apply Lemma \ref{L:ActionAT} to compute the $\phi$-action on $[(\Theta_P,\chi_P)]$. To avoid confusion, we use $P'$ to denote the image of $P$ under the projection from $\mathbb A^5$ to $\mathbb A^3$, that is, $P'=(c,0,e)$. Note by Remark \ref{R:A5} (i), we get $\chi_P=\chi_{P'}$. Also we denote by $\phi'$ the image of $\phi=(\gamma,G)$ in $\text{GL}_3(\field)$ by Equation \eqref{E:EmA3}. Hence, the image of $[(\Theta_P,\chi_P)]$ under the $\phi$-action is given by  
\begin{align}\label{E:phi-action}
\begin{split}
\left[(\gamma^p\, \phi_2(\Theta_P),\, \gamma\, (\phi_2\otimes\phi_2)(\chi_P))\right]=&\ \left[(\gamma^p\, \phi_2(\Theta_P),\, \gamma\, (\phi_2\otimes\phi_2)(\chi_{P'}))\right]\\ 
=&\ \left[((\gamma^p\, \phi_2(\Theta_P)+\Phi_z(\sigma),\, \chi_{\phi'(P')})\right].
\end{split}
\end{align}
It follows from $P'=(c,0,e)$ and Lemma \ref{L:ActionAT}, 
\[\sigma=\frac{1}{2}\gamma\, c\,\left(g_{11}g_{21}\,x^2+2g_{12}g_{21}\, xy+g_{12}g_{22}\, y^2\right)=\frac{1}{2}\gamma\, cg_{12}g_{22}\, y^2.\] 
When $\T=(\T9)$, by Table~\ref{tab:rank 2}, we have $z^{[p]}=0$, $x^{[p]}=y,\, y^{[p]}=0$ and $\rho_z=0$. Then, $\Phi_z(\sigma)=\sigma^p +\rho_z^{p-1}(\sigma)=0$. By Equation~\eqref{E:phi-action}, $\phi: [(\Theta_P,\chi_P)]\to [(\Theta_Q,\chi_Q)]$ with
\begin{align*}
Q=\left(\gamma^p\, ag_{11},\, \gamma^p\, ag_{12},\, \gamma\, g_{11}g_{22}c,\, 0,\, \gamma^{\frac{1}{p}}\,g_{22}c\right).
\end{align*}
Suppose $\gamma=1$ and $g_{11}=g_{22}=1$ in $\phi=(\gamma,G)$. Then $Q=(a,\,ag_{12},\, c,\, 0,\ e)$. In this case, since $P$ is the projection of $Q$, we get $P \sim Q$ in $B^+(\T)$ by Lemma \ref{L:BIJ}. By Definitions~\ref{D:EquiR} and \ref{D:equiv}, $[(\Theta_P,\chi_P)] \sim [(\Theta_Q,\chi_Q)]$ and the $\phi$-action is trivial. Consequently, the subgroup $\{\phi=(\gamma,G)\in \Aut(\T)\, |\, \gamma=1, g_{11}=g_{22}=1\}$ of $\Aut(\T)$ acts trivially on $\mathcal H^2(\T)$, whose quotient group is isomorphic to $\widetilde{\Aut}(\T)$. Therefore, we can always assume $\phi\in \widetilde{\Aut}(\T)$. Now, it is clear that $\phi: P\to Q=\phi(P)$ with $g_{12}=0$ by Equation \eqref{E:ActionG2}, which completes the proof.    
\end{proof}

For each type $\T\in \mathcal E$, we can compute the $\widetilde{\Aut}(\T)$-orbits using the group action \eqref{E:ActionG2}, and identify the isomorphism classes of quantum $p$-groups in $\mathscr X(2)$ of type $\T$ by applying Theorem~\ref{TH:HopfO} and Theorem \ref{T:NPerT}. Finally, these quantum $p$-groups can be constructed by using PD data. Classification of such quantum $p$-groups is given in Table \ref{tab:WConC}.

%move this paragraph to before Table to make it fit into page
The results in Table \ref{tab:WConC}  can be verified briefly as follows. First of all, we can check that all points $P$ listed in Table \ref{tab:WConC}, including infinite families, lie in different $\widetilde{\Aut}(\T)$-orbits. Then it suffices to show that every point and infinite family in $\mathscr B^+(\T)$ must lie in some orbit $[P]$. Now let $\phi=(\gamma,G)\in \widetilde{\Aut}(\T)$, and $P\in \mathscr B^+(\T)$. Consider the restriction of $\phi$ on the last three coordinates of $P$. We see that points in $\mathscr B^+(\T)$, as representatives of orbits, have certain explicit last coordinates. For example, when $\T=(\T1)$, the representatives of orbits in $\mathscr B^+(\T)$ can be chosen from the forms $(a,b,1,0,0),\, (a,b,0,1,0)$ and $(a,b,1,1,0)$ for some $a,b\in \field$. It is clear that points of different forms can not belong to the same orbit. It then remains to distinguish orbits one from another within a given form, which can be done by type; see Appendix~\ref{Appen:WConC}.

\newpage
\FloatBarrier
\begin{table}[!htp] % table is a floating environment. TeX has special rules for placing floats so that they cause minimal disruption of the surrounding text. You can give LaTeX hints of how to place your floats. For example, to make LaTeX try to place it where you inserted it, specify the option h (for ?here?). Note that this won?t always work because LaTeX still tries to make the text look nice. Your hint is only that: a hint, not a definite command. You can force LateX to obey your command to place a floating environment HERE by loading the package "float" and specifying the option H (capital letter!). But notice that this can lead to an ugly layout.
\caption {Quantum $p$-groups of nonpermissible types} \label{tab:WConC} \vspace{-0.2cm}
\begin{center}
\begin{tabular}{| c | p{3.1cm} | p{7.2cm} | p{2cm} |}  %{|l|l|l|l|}  %
\hline
\bf{Types}  & $[P]$ & $[(\Psi_P+\Theta_P,\chi_P)]$ & \bf{\# of Iso. classes} \\
\hline \hline
(\T1)  &  $P=(0,0,1,0,0)$ & $\Psi_P+\Theta_P=0,\, \chi_P=x\otimes y$ & 1 \\
      &  $P=(1,0,1,0,0)$ & $\Psi_P+\Theta_P=x,\, \chi_P=x\otimes y$ & 1 \\
      &  $P=(0,0,0,1,0)$ & $\Psi_P+\Theta_P=0,\, \chi_P=\Bock(x)$ & 1 \\
      &  $P=(1,0,0,1,0)$ & $\Psi_P+\Theta_P=x,\, \chi_P=\Bock(x)$ & 1 \\
      &  $P=(0,1,0,1,0)$ & $\Psi_P+\Theta_P=y,\, \chi_P=\Bock(x)$ &  1\\
      &  $P=(0,0,1,1,0)$ & $\Psi_P+\Theta_P=0,\, \chi_P=x\otimes y+\Bock(x)$ &1  \\
       &  $P=(1,0,1,1,0)$ & $\Psi_P+\Theta_P=x,\, \chi_P=x\otimes y+\Bock(x)$ & 1 \\
       &  $P=(0,1,1,1,0)$ & $\Psi_P+\Theta_P=y,\, \chi_P=x\otimes y+\Bock(x)$ &1  \\
\hline
(\T2)  & $P=(0,0,1,0,0)$ & $\Psi_P+\Theta_P=0,\, \chi_P=x\otimes y$ &1  \\
     & $P=(0,1,1,0,0)$ & $\Psi_P+\Theta_P=y,\, \chi_P=x\otimes y$ & 1 \\
     & $P=(0,0,0,1,0)$ & $\Psi_P+\Theta_P=xy^{p-1},\, \chi_P=\Bock(x)$ & 1 \\
     & $P=(0,1,0,1,0)$ & $\Psi_P+\Theta_P=xy^{p-1}+y,\, \chi_P=\Bock(x)$ & 1 \\
      & $P=(0,0,0,0,1)$ & $\Psi_P+\Theta_P=0,\, \chi_P=\Bock(y)$ & 1 \\
       & $P=(0,1,0,0,1)$ & $\Psi_P+\Theta_P=y,\, \chi_P=\Bock(y)$ & 1 \\
%         & $(0,0,1,1,0)$ & $\Theta=xy^{p-1},\chi=x\otimes y+\Bock(y)$ &  \\
%       & $(0,0,1,0,1)$ & $\Theta=0,\chi=x\otimes y+\Bock(x)$ &  \\
       & $P=(0,\xi,1,1,0)$ & $\Psi_P+\Theta_P=xy^{p-1}+\xi y,\,\chi_P=x\otimes y+\Bock(x)$ & $\k/\mu_2$ \\
        & $P=(0,\xi,1,0,1)$ & $\Psi_P+\Theta_P=\xi y,\, \chi_P=x\otimes y+\Bock(y)$ &  $\k$\\
\hline
(\T4)  &  $P=(0,0,1,0,0)$  &  $\Psi_P+\Theta_P=0,\,\chi_P=x\otimes y$ &1  \\
      &  $P=(0,1,1,0,0)$  &  $\Psi_P+\Theta_P=y,\, \chi_P=x\otimes y$ & 1 \\
      &  $P=(0,0,0,0,1)$  &  $\Psi_P+\Theta_P=0,\,\chi_P=\Bock(y)$ &  1\\
      &  $P=(0,1,0,0,1)$  &  $\Psi_P+\Theta_P=y,\,\chi_P=\Bock(y)$ & 1 \\
      &  $P=(0,\xi,1,0,1)$  &  $\Psi_P+\Theta_P=\xi y,\,\chi_P=x\otimes y+\Bock(y)$ & $\k/\mu_{(p-1)/2}$  \\
\hline
(\T9) & $P=(0,0,1,0,0)$ & $\Psi_P+\Theta_P=0,\,\chi_P=x\otimes y$ & 1 \\
     & $P=(1,0,1,0,0)$ & $\Psi_P+\Theta_P=x,\, \chi_P=x\otimes y$ &1 \\
     & $P=(0,0,0,0,1)$ & $\Psi_P+\Theta_P=0,\,\chi_P=\Bock(y)$ & 1\\
     & $P=(1,0,0,0,1)$ &  $\Psi_P+\Theta_P=x,\,\chi_P=\Bock(y)$ &  1\\
     & $P=(\xi,0,1,0,1)$ & $\Psi_P+\Theta_P=\xi x,\,\chi_P=x\otimes y+\Bock(y)$ & $\k/\mu_{p^2-p-1}$\\
\hline
\end{tabular}
\end{center}
\end{table}

%%%%%%%%%%%%%%%%%%%%%%%%%%%%%%%%

\section{Closing discussions}
\label{close}

\noindent
{\bf Primitive deformations for $p=2$.} Throughout Sections~\ref{sec:realization} and \ref{sec:X(2)}, we set $p>2$. The primitive deformation in characteristic $p=2$ employs a similar technique except we have to use the other description of the cohomology ring for $p=2$, as seen in Lemma \ref{L:HochUH}. As a consequence, the second cohomology group $\HL^2\left(\Omega\, u(\mathfrak h)\right)$ together with its $\Aut(\T)$-action, when realized in the affine space, are different from our Lemma~\ref{L:AutAction} and Remark~\ref{R:H2A3}. We leave the case $p=2$ to readers who could utilize our technique to pursue this case for their own interest, e.g., the classification of quantum $p$-groups of dimension $8$ in characteristic $2$. 
 \vspace{0.2cm}
 
\noindent
{\bf Questions.} We close this paper by proposing some questions to be considered in future projects:
\begin{enumerate}
\item Hopf quivers of Nichols algebras in positive characteristic are studied in \cite{CLW, OZ}. The $p^3$-dimensional connected Hopf algebras $H$ classified in our sequel (cf.~\cite{NWW1}) are given by explicit generators and relations. This makes it very convenient to associate quivers $Q$ to those $H$ whose algebra structures are basic (i.e. the quotient algebra modulo the Jacobson radical is a copy of the base field $\k$). How do these related quivers look like? 

 \vspace{0.2cm}
 \item In Section~\ref{sec:X(2)}, we obtain many infinite parametric families of finite-dimensional Hopf algebras in positive characteristic. Note that in \cite[Theorem 2.2]{NgSc}, two finite-dimensional Hopf algebras over $\k$ are gauge equivalent if their finite representations are $\k$-linear tensor equivalent. We want to know whether Hopf algebras of each parametric family in Section~\ref{sec:X(2)} are all gauge equivalent?
 
 \vspace{0.2cm}
 % \item We know the representation of the (resp. restricted) universal enveloping algebra of a (resp. restricted) Lie algebra is equivalent to the (resp. restricted) representation of the Lie algebra. This applies to all the $p^3$-dimensional Hopf algebras listed as C types in \cite{NWW1} or Table~\ref{tab:typeC}. Moreover, this applies to any Hopf algebras which are isomorphic, as algebras, to restricted universal enveloping algebras. Hence, we want to know in our classification which $p^3$-dimensional connected Hopf algebras have distinct algebra structures from restricted universal enveloping algebras, and what are their representations? 
 \item In this paper, we provide quite a few examples of non-commutative and non-cocommutative Hopf algebras in positive characteristic (e.g. those non-commutative algebras in Tables~\ref{tab:ConC} and \ref{tab:WConC} with element $\chi$ having nonzero coefficient in the $x\otimes y$ term). What are the cohomology rings $\HL^*(H,\k):=\Ext_H^*(\k,\k)$ of such Hopf algebras?
\end{enumerate}

%%%%%%%%%%%%%%%%%%%%%%%%%%%%%%%%

\appendix 

%%%%%%%%%%%%%%%%%%%%%%%%%%%%%%%%

\section{Classifications of $p^3$-dimensional \\ connected Hopf algebras} 
\label{Appen:complete classification}

Let $H$ be a $p^3$-dimensional connected Hopf algebra over an algebraically closed field $\k$ of characteristic $p>0$. Let $u(\Prim(H))$ be the Hopf subalgebra of $H$ generated by all primitive elements. The isomorphism classes of $H$ will be presented by a quotient of the free algebra $\k\langle x,y,z \rangle /I$. The defining relation $I$ and the comultiplication are provided in terms of the generators $x$, $y$ and $z$. We follow the notation in \cite{Hen} to write the reduced comultiplication of $H$ by $\psi$, for example, $\psi(x)=\Delta(x)-x\otimes 1-1\otimes x$. Throughout, we denote the expressions 
\begin{align*}
\Bock(r)&:=\sum_{1\le i\le p-1}\frac{(p-1)!}{i!\,(p-i)!}\, \left(r^i\otimes r^{p-i}\right), \quad \text{ for any } r \in H, \\ 
\mathcal Z &:=\Bock(x)[y\otimes 1+1\otimes y+\Bock(x)]^{p-1}+\Bock(y), \text{ and} \\
\mathcal Z' &:=\Bock(x)(y\otimes 1+1\otimes y)^{p-1}+\Bock(y).
\end{align*}
By Nichols-Zoeller Theorem, dimension of $u(\Prim(H))$ divides dimension of $H$. Hence, our strategy is to consider all possible dimensions of $u(\Prim(H))$ when $H$ is $p^3$-dimensional. The classifications in Tables~\ref{tab:typeA}, \ref{tab:typeB}, and \ref{tab:typeC} were detailed in our preceding paper \cite{NWW1}, where we classified the types A, B, and C according to the dimensions $p$, $p^2$ (non-commutative), and $p^3$ of $u(\Prim(H))$, respectively. For each type, we provide the isomorphism classes, their algebra and coalgebra structures, and indicate whether they are commutative, semisimple, or local. The main results in Section~\ref{sec:X(2)} classify those in Table \ref{tab:typeT}, where $u(\Prim(H))$ is a $p^2$-dimensional commutative Hopf subalgebra of $H$. 

\newpage
\FloatBarrier
\begin{table}[!htp]%[h] %[htp] % table is a floating environment. TeX has special rules for placing floats so that they cause minimal disruption of the surrounding text. You can give LaTeX hints of how to place your floats. For example, to make LaTeX try to place it where you inserted it, specify the option h (for ?here?). Note that this won't always work because LaTeX still tries to make the text look nice. Your hint is only that: a hint, not a definite command. You can force LateX to obey your command to place a floating environment HERE by loading the package "float" and specifying the option H (capital letter!). But notice that this can lead to an ugly layout. 
\begin{center}%\small\setlength\tabcolsep{2pt}
\caption{Connected $p^3$-dim Hopf algebras, when $\dim u(\Prim(H)) = p$}  \label{tab:typeA} \vspace{-0.2cm}
\begin{tabular}{| c | c | p{5cm} | p{4.5cm} | c |} %{| l| l | l  | l  |l|} %{l  c c c c c c c c c c c c c c c c c c c c c c c  }
\hline       % \toprule
%\begin{tabular}{cc p{5.2cm}llc}
%\toprule
%\bf{Type} &\bf{Comm.} & \bf{Alg. Structure}       & \bf{Coalg. Structure} & \bf{S.S./Loc.} \\ \toprule
%\begin{tabular}{| l| l | l  | l  |l|}
%\hline
\bf{Type} & \bf{Comm.} & \bf{Alg. Structure}       & \bf{Coalg. Structure} & \bf{S.S./Loc.} \\
\hline   \hline  
%\toprule
A1    &  y &$x^p-x=y^p-y=z^p-z=0$   &$\psi(x)=0$, $\psi(y)=\Bock(x)$, $\psi(z)=\mathcal Z$  & y/n \\ \hline
A2    & y &$x^p=y^p-x=z^p-y=0$   &$\psi(x)=0$, $\psi(y)=\Bock(x)$, $\psi(z)=\mathcal Z'$ & n/y \\ \hline
A3    & y &$x^p=y^p=z^p=0$   &$\psi(x)=0$, $\psi(y)=\Bock(x)$, $\psi(z)=\mathcal Z'$ & n/y\\ \hline
A4    & y&$x^p=y^p=z^p-x=0$   &$\psi(x)=0$, $\psi(y)=\Bock(x)$, $\psi(z)=\mathcal Z'$ & n/y\\ \hline
A($\lambda$)    &n  &$x^p=y^p=z^p+x^{p-1}y-\lambda x=0, [x,y]=[x,z]=[y,z]-x=0$   &$\psi(x)=0$, $\psi(y)=\Bock(x)$, $\psi(z)=\mathcal Z'$& \multirow{2}*{n/y} \\ \hline
\end{tabular}
\end{center}
\end{table}
For the parametric family A($\lambda$), we can choose $\lambda=0$ when $p=2$. When $p>2$, A$(\lambda)\cong$ A$(\lambda')$ iff $\lambda=\gamma \lambda'$ for some $\gamma\in \sqrt[p^2+p-1]{1}$.

\FloatBarrier
\begin{table}[!htp]%[h]%[htp] % table is a floating environment. TeX has special rules for placing floats so that they cause minimal disruption of the surrounding text. You can give LaTeX hints of how to place your floats. For example, to make LaTeX try to place it where you inserted it, specify the option h (for ?here?). Note that this won't always work because LaTeX still tries to make the text look nice. Your hint is only that: a hint, not a definite command. You can force LateX to obey your command to place a floating environment HERE by loading the package "float" and specifying the option H (capital letter!). But notice that this can lead to an ugly layout. 
\begin{center}%\small\setlength\tabcolsep{2pt}
\caption{Connected $p^3$-dim Hopf algebras, when $\dim u(\Prim(H))= p^2$ and $u(\Prim(H))$ is non-commutative}\label{tab:typeB} \vspace{-0.2cm}
\begin{tabular}{| c | c | p{5.1cm} | p{4.5cm} | c |} 
\hline 
\bf{Type} & \bf{Comm.} & \bf{Alg. Structure}       & \bf{Coalg. Structure} & \bf{S.S./Loc.} \\
\hline   \hline
B1     & n & $x^p-x=y^p=z^p=0, [x,y]-y=[x,z]=[y,z]=0$   &$\psi(x)=\psi(y)=0$, $\psi(z)=\Bock(y)$ & n/n\\ \hline
B2 & n& $x^p-x=y^p=z^p-z=0, [x,y]-y=[x,z]=[y,z]-yf(x)=0$  &$\psi(x)=\psi(y)=0$, $\psi(z)=\Bock(x)$ & n/n\\ \hline
B3 &n  & $x^p-x=y^p=z^p=0, [x,y]-y=[x,z]-z=[y,z]-y^2=0$  &$\psi(x)=\psi(y)=0$, $\psi(z)=-2x\otimes y$ & n/n\\ \hline
\end{tabular}
\end{center}
\end{table}
(B3) only appears in $p>2$. In the algebra structure of (B2), we let 
$$f(x)=\sum_{i=1}^{p-1} (-1)^{i-1}(p-i)^{-1}x^i.$$ \\

\FloatBarrier
\begin{table}[!htp]%[h] %[htp] % table is a floating environment. TeX has special rules for placing floats so that they cause minimal disruption of the surrounding text. You can give LaTeX hints of how to place your floats. For example, to make LaTeX try to place it where you inserted it, specify the option h (for ?here?). Note that this won't always work because LaTeX still tries to make the text look nice. Your hint is only that: a hint, not a definite command. You can force LateX to obey your command to place a floating environment HERE by loading the package "float" and specifying the option H (capital letter!). But notice that this can lead to an ugly layout. 
\begin{center}%\small\setlength\tabcolsep{2pt}
\caption{Connected $p^3$-dim Hopf algebras, when $\dim u(\Prim(H)) = p^3$} \label{tab:typeC} \vspace{-0.2cm}
\begin{tabular}{| c | c | p{5cm} | p{4.5cm} | c |} 
\hline
\bf{Type} & \bf{Comm.} & \bf{Alg. Structure}       & \bf{Coalg. Structure} & \bf{S.S./Loc.} \\
\hline    \hline  
C1   & y & $x^p-x=y^p-y=z^p-z=0$   &$\psi(x)=\psi(y)=\psi(z)=0$ & y/n\\ \hline
C2   & y& $x^p-y=y^p-z=z^p=0$   &$\psi(x)=\psi(y)=\psi(z)=0$ & n/y\\ \hline
C3     &y& $x^p=y^p-z=z^p=0$   &$\psi(x)=\psi(y)=\psi(z)=0$ & n/y\\ \hline
C4   &y& $x^p=y^p=z^p=0$   &$\psi(x)=\psi(y)=\psi(z)=0$ & n/y\\ \hline
C5  &n& $x^p=y^p=z^p=0, [x,y]-z=[x,z]=[y,z]=0$   &$\psi(x)=\psi(y)=\psi(z)=0$ & n/y\\ \hline
C6  &n& $x^p-z=y^p=z^p=0, [x,y]-z=[x,z]=[y,z]=0$   &$\psi(x)=\psi(y)=\psi(z)=0$ & n/y\\ \hline
C7   & y & $x^p=y^p=z^p-z=0$   &$\psi(x)=\psi(y)=\psi(z)=0$ & n/n\\ \hline
C8 &y& $x^p-y=y^p=z^p-z=0$   &$\psi(x)=\psi(y)=\psi(z)=0$ & n/n\\  \hline
C9    &y& $x^p=y^p-y=z^p-z=0$   &$\psi(x)=\psi(y)=\psi(z)=0$ & n/n\\ \hline
C10  &n& $x^p=y^p=z^p-z=0, [x,y]-z=[x,z]=[y,z]=0$    &$\psi(x)=\psi(y)=\psi(z)=0$ & n/n\\ \hline
C11    &n& $x^p-x=y^p=z^p=0, [x,y]-y=[x,z]=[y,z]=0$   &$\psi(x)=\psi(y)=\psi(z)=0$ & n/n\\ \hline
C12 &n& $x^p-x=y^p-z=z^p=0, [x,y]-y=[x,z]=[y,z]=0$   &$\psi(x)=\psi(y)=\psi(z)=0$ & n/n\\ \hline
C13     &n& $x^p-x=y^p=z^p-z=0, [x,y]-y=[x,z]=[y,z]=0$   &$\psi(x)=\psi(y)=\psi(z)=0$ & n/n\\ \hline
C14   &n& $x^p-x=y^p-z=z^p-z=0, [x,y]-y=[x,z]=[y,z]=0$   &$\psi(x)=\psi(y)=\psi(z)=0$ & n/n\\ \hline
C15    &n& $x^p=y^p=z^p-z=0, [x,y]-z=[x,z]-x=[y,z]+y=0$   &$\psi(x)=\psi(y)=\psi(z)=0$ & n/n\\ \hline
C$(\lambda,\delta)$ &n& $x^p=y^p=z^p-\delta z=0, [x,y]=[x,z]-\lambda x=[y,z]-\lambda^{-1}y=0$   &$\psi(x)=\psi(y)=\psi(z)=0$ & n/n\\ \hline
\end{tabular}
\end{center}
\end{table}
(C15) only appears in $p>2$ and C$(\lambda,\delta)$ satisfies $\lambda^{p-1}=\delta=\pm 1$. We also have C$(\lambda,\delta)\cong$ C$(\lambda',\delta')$ iff $\delta=\delta'$ and $\lambda=\lambda'$ or $\lambda\cdot\lambda'=1$.

%{\bf Classification results for $\mathscr X(2)$, where $p>2$.}
%Considering possible algebra structures for Hopf algebras $H \in \mathscr X(2)$, there are totally 14 pre-classified types and one parametric family. Each one breaks down into smaller isomorphism classes of $H$.

%\newpage
\FloatBarrier
\begin{table}[!htp]%[h]%[htp] % table is a floating environment. TeX has special rules for placing floats so that they cause minimal disruption of the surrounding text. You can give LaTeX hints of how to place your floats. For example, to make LaTeX try to place it where you inserted it, specify the option h (for ?here?). Note that this won't always work because LaTeX still tries to make the text look nice. Your hint is only that: a hint, not a definite command. You can force LateX to obey your command to place a floating environment HERE by loading the package "float" and specifying the option H (capital letter!). But notice that this can lead to an ugly layout. 
\begin{center}%\small\setlength\tabcolsep{2pt}
\caption{Connected $p^3$-dim Hopf algebras, when $p>2$, $\dim u(\Prim(H)) = p^2$ and $u(\Prim(H))$ is commutative}\label{tab:typeT} \vspace{-0.2cm}
\begin{tabular}{| c | p{6cm} | p{3.1cm} | p{1.9cm} | p{1.7cm} |} 
\hline
\bf{Type} & \bf{Alg. Structure}  & \bf{Coalg. Structure} & \bf{\# of Iso. classes} & \bf{Results} \\
\hline   \hline  
T1  & $x^p=y^p=z^p+\Theta=0, [x,y]=[z,x]=[z,y]=0$   &$\psi(x)=\psi(y)=0$, $\psi(z)=\chi$   & 8 classes &  Table \ref{tab:WConC}  \\  \hline
T2  & $x^p=y^p=z^p+\Theta=0, [x,y]=[z,x]-y=[z,y]=0$   &$\psi(x)=\psi(y)=0$, $\psi(z)=\chi$   & 6 classes \& 2 families &  Table \ref{tab:WConC}  \\  \hline
T3 &  $x^p=y^p=z^p-z+\Theta=0, [x,y]=[z,x]=[z,y]=0$  &$\psi(x)=\psi(y)=0$, $\psi(z)=\chi$   & none &  Table \ref{tab:rank 2}  \\  \hline
T4 &  $x^p-x=y^p=z^p+\Theta=0, [x,y]=[z,x]=[z,y]=0$   &$\psi(x)=\psi(y)=0$, $\psi(z)=\chi$   & 4 classes \& 1 family &  Table \ref{tab:WConC}  \\  \hline
T5 &  $x^p-x=y^p=z^p+\Theta=0, [x,y]=[z,x]=[z,y]-x=0$   &$\psi(x)=\psi(y)=0$, $\psi(z)=\chi$   & 1 class \& 1 family &  Table \ref{tab:ConC}  \\  \hline
T6 &  $x^p-x=y^p=z^p-z+\Theta=0, [x,y]=[z,x]=[z,y]=0$   &$\psi(x)=\psi(y)=0$, $\psi(z)=\chi$   & 1 class &  Table \ref{tab:ConC}  \\  \hline
T7 &  $x^p-x=y^p=z^p-z+\Theta=0, [x,y]=[z,x]=[z,y]-y=0$  &$\psi(x)=\psi(y)=0$, $\psi(z)=\chi$   & 3 classes &  Table \ref{tab:ConC}  \\  \hline
T8 &  $x^p-x=y^p=z^p-z+\Theta=0, [x,y]=[z,x]=[z,y]-x-y=0$  &$\psi(x)=\psi(y)=0$, $\psi(z)=\chi$   & 2 families &  Table \ref{tab:ConC}  \\  \hline
T9 &  $x^p-y=y^p=z^p+\Theta=0, [x,y]=[z,x]=[z,y]=0$  &$\psi(x)=\psi(y)=0$, $\psi(z)=\chi$   & 4 classes \& 1 family &  Table \ref{tab:WConC}  \\  \hline
T10 &  $x^p-y=y^p=z^p+\Theta=0, [x,y]=[z,x]-y=[z,y]=0$  &$\psi(x)=\psi(y)=0$, $\psi(z)=\chi$   & 1 class \& 1 family &  Table \ref{tab:ConC}  \\  \hline
T11 &  $x^p-y=y^p=z^p-z+\Theta=0, [x,y]=[z,x]=[z,y]=0$  &$\psi(x)=\psi(y)=0$, $\psi(z)=\chi$   & none &  Table \ref{tab:rank 2}  \\  \hline
T12 & $x^p-y=y^p=z^p-z+\Theta=0, [x,y]=[z,x]-x=[z,y]=0$  &$\psi(x)=\psi(y)=0$, $\psi(z)=\chi$   & 1 class &  Table \ref{tab:ConC}  \\  \hline
T13 &  $x^p-x=y^p-y=z^p+\Theta=0, [x,y]=[z,x]=[z,y]=0$  &$\psi(x)=\psi(y)=0$, $\psi(z)=\chi$   & none &  Table \ref{tab:rank 2}  \\  \hline
T14 &  $x^p-x=y^p-y=z^p-z+\Theta=0, [x,y]=[z,x]=[z,y]=0$   &$\psi(x)=\psi(y)=0$, $\psi(z)=\chi$   & 3 classes &  Table \ref{tab:ConC}  \\  \hline
T($\zeta$) & $x^p=y^p=z^p-z+\Theta=0, [x,y]=[z,x]-x=[z,y]-\zeta y=0$, for $\zeta\in \mathbb F_p \setminus \{-1\}$ &$\psi(x)=\psi(y)=0$, $\psi(z)=\chi$   & $(p+1)/2$ classes  &  Table \ref{tab:ConC}  \\  \hline
\end{tabular}
\end{center}
\end{table}

Hence, these tables provide the complete classification of $p^3$-dimensional connected Hopf algebras over an algebraically closed field of positive characteristic.

%%%%%%%%%%%%%%%%%%%%%%%%%%%%%%%%
\newpage  % remove this if the spacing works out

\section{Verifications for Section~\ref{sec:X(2)}}

\subsection{Verification of Table \ref{tab:ConC} for $\T=(\T5)$}\label{Appen:ConC}
Note that $z^{[p]}=0, x^{[p]}=x, y^{[p]}=0$, and $\rho_z(x)=0,\rho_z(y)=x$. Let $P=(a,b,c)\in \mathbb A^3\setminus\{0\}$. By Proposition~\ref{P:Cobar} (iii), a direct computation in $\HL^2(\Omega\, u(\mathfrak h))$ shows that  
\begin{align*}
&\ [\Phi_z(\chi_P)]\\
=&\ [\chi_P^p+\rho_z^{p-1}(\chi_P)]\\
=&\ [(a\, x\otimes y+\Bock(b\, x+c\, y))^p+\rho_z^{p-1}(a\, x\otimes y)+\rho_z^{p-1}\Bock(b\, x+c\, y)]\\
=&\ [a^p\, x^p\otimes y^p+\Bock(b^p\, x^p+c^p\, y^p)+a\,\rho_z^{p-2}(\rho_z(x)\otimes y+x\otimes \rho_z(y))+\rho_z^{p-1}\Bock(b\, x+c\, y)]\\
=&\ [\Bock(b^p\, x)+a\,\rho_z^{p-2}(x\otimes x)]+[\rho_z^{p-1}\Bock(b\, x+c\, y)]\\
=&\ [\Bock(b^p\, x)]+[\rho_z^{p-1}\Bock(b\, x+c\, y)]\\
=&\ [\Bock(b^p\, x)].
\end{align*}
If $P\in \mathscr A^+(\T)$, one sees that $b=0$ since $[\Phi_z(\chi_P)]=0$ by Remark \ref{R:CCzero}. 

Now, let $P=(a,0,c)\in \mathbb A^3\setminus\{0\}$. Thus, $\Phi_z(\chi_P)$ can be written as 
\begin{align*}
\Phi_z(\chi_P)=&\ \rho_z^{p-1}\Bock(c\, y)\\
=&\ \partial^1\left(-\sum_{i_1+\cdots+i_p=p-2}\, \frac{(p-2)!}{i_1!\cdots i_p!}\, \rho_z^{i_1}(c\, y)\cdots \rho_z^{i_{p-1}}(c\, y)\rho_z^{1+i_p}(c\, y)\right)\\
=&\ \partial^1\left(-\sum_{i_1+\cdots+i_p=p-2}\, \frac{(p-2)!}{i_1!\cdots i_p!}\, c^p\, \rho_z^{i_1}(y)\cdots \rho_z^{i_{p-1}}(y)\rho_z^{i_p}(x)\right).
\end{align*}
Since $\rho_z(y)=x$ and $\rho_z^i(y)=0$ for $i\ge 2$, nonzero terms occur in the above summation only if $i_p=0$ and, for $1\le k\le p-1$, all $i_k$'s equal to $1$ except one of the $i_k$'s is zero. Hence,  
\begin{align*}
\Phi_z(\chi_P)=\partial^1(-(p-1)(p-2)!\, c^p\, x^{p-1}y)=\partial^1(c^p\, x^{p-1}y).
\end{align*}
Set $\Theta=c^p(x^{p-1}y-y)\in u(\mathfrak h)^+$. Then $\Phi_z(\chi_P)=\partial^1(\Theta)$ since $\partial^1(y)=0$ and $\rho_z(\Theta)=\rho_z(c^p(x^{p-1}y-y))=c^p(x^{p-1}\rho_z(y)-\rho_z(y))=c^p(x^p-x)=0$. By Proposition \ref{P:RCdata}, $(\Theta,\chi_P)$ is PD datum and $P\in \mathscr A^+(\T)$ by Remark \ref{R:APCD}. In conclusion, $\mathscr A^+(\T)$ contains all points $P=(a,0,c)\in \mathbb A^3\setminus\{0\}$. 

Next, we compute the $\Aut(\T)$-action on $\mathscr A^+(\T)$. Let $\phi=(\gamma,G)\in \Aut(\T)$. Since $R=e_{11}$, $M=e_{21}$ and $\lambda=\delta=0$, by conditions \eqref{E:RepC}, $G=\text{diag}(\alpha,\beta)$ and $\beta=\alpha\gamma$, where $\alpha,\beta,\gamma$ are nonzero scalars with $\alpha^p=\alpha$. According to Equation \eqref{E:EmA3}, the $\phi$-action on $P=(a,0,c)\in \mathscr A^+(\T)$ is given by
\begin{align*}
\phi(a,0,c)=(\gamma^2\alpha^2\, a,\, 0,\, \gamma^{\frac{1+p}{p}}\alpha\ c).
\end{align*}
A simple calculation shows that the $\Aut(\T)$-orbits of $\mathscr A^+(\T)$ contains one single point and one quotient line. The single point can be represented by $(1,0,0)$ and the quotient line is in terms of $(\xi,0,1)$, where two points given by $(\xi,0,1)$ and $(\xi',0,1)$ are in the same orbit if and only if $\xi=\tau\xi'$ for some $(p^2-1)/2$-th root of unity $\tau$. Furthermore, by Theorem \ref{T:BPT}, orbits $[(1,0,0)]$ and $[(\xi,0,1)]$ correspond to $[(0,x\otimes y)]$ and $[(x^{p-1}y-y,\xi x\otimes y+\Bock(y))]$ in $\mathcal H^2(\T)$, respectively.
\\

%%%%%%%%%%%%%%%%%%%%%%%%%%%%%%%%

\subsection{Verification of Table \ref{tab:SetWC} for $\T=(\T2)$} \label{Appen:SetWC}
Note that $z^{[p]}=0, x^{[p]}=y^{[p]}=0$, and $\rho_z(x)=y,\rho_z(y)=0$. Let $P=(a,b,c,d,e)\in \mathbb A^2\times \mathbb A^3$. By Proposition \ref{P:Cobar} (iii),
\begin{align*}
&\ \Phi_z(\chi_P)\\
=&\ \rho_z^{p-1}\Bock(d\, x+e\, y)\\
=&\ \partial^1\left(-\sum_{i_1+\cdots+i_p=p-2}\ \frac{(p-2)!}{i_1!\cdots i_p!}\, \rho_z^{i_1}(d\, x+e\, y)\cdots \rho_z^{i_{p-1}}(d\, x+e\, y)\rho_z^{1+i_p}(d\, x+e\, y)\right)\\
=&\ \partial^1\left(-\sum_{i_1+\cdots+i_p=p-2}\ \frac{(p-2)!}{i_1!\cdots i_p!}\, \rho_z^{i_1}(d\, x)\cdots \rho_z^{i_{p-1}}(d\, x)\rho_z^{1+i_p}(d\, x)\right)\\
=&\ \partial^1\left(-\sum_{i_1+\cdots+i_p=p-2}\ \frac{(p-2)!}{i_1!\cdots i_p!}\, d^p\, \rho_z^{i_1}(x)\cdots \rho_z^{i_{p-1}}(x)\rho_z^{i_p}(y)\right)\\
=&\ \partial^1(-(p-1)(p-2)!\, d^p\, xy^{p-1})\\
=&\ \partial^1(d^p\, xy^{p-1}).
\end{align*}
Then,
\begin{align*}
\Phi_z(\chi_P)=&\ \chi_P^p-\delta\chi_P+\rho_z^{p-1}(\chi_P)\\
=&\ \rho_z^{p-1}(c\, x\otimes y)+\rho_z^{p-1}\Bock(d\, x+e\, y)\\
=&\ c\, \rho_z^{p-2}(\rho_z(x)\otimes y+x\otimes \rho_z(y))+\partial^1(d^p\, xy^{p-1})\\
=&\ c\, \rho_z^{p-2}(y\otimes y)+\partial^1(d^p\, xy^{p-1})\\
=&\ \partial^2(d^p\, xy^{p-1}).
\end{align*}
Suppose $P\in \mathscr B^+(\T)$. By Definition \ref{D:BSet}, we set $\Psi_P=d^p\, xy^{p-1}$. Since $\rho_z(d^p\, xy^{p-1}+a\, x+b\, y)=d^p\, y^p+a\, y=a\, y=0$, one sees that $a=0$. Hence, the set $\mathscr B^+(\T)$ contains all points $P=(0,b,c,d,e)$ with $\Psi_P=d^p\, xy^{p-1}$, where $c,d,e$ are not all zero. 
\\

%%%%%%%%%%%%%%%%%%%%%%%%%%%%%%%%

\subsection{Verification of Table \ref{tab:WConC}}\label{Appen:WConC}

\bf{Case (\T=\T1):}\rm\ The representative points $P\in \mathscr B^+(\T)$ have one of the following forms:
\[
(a,b,1,0,0),\ (a,b,0,1,0),\ (a,b,1,1,0),\ \mbox{for some}\ a,b\in \k.
\]

Let $P=(a,b,1,0,0)$. If $(a,b)=(0, 0)$, then $P=(0,0,1,0,0)$. Now suppose $(a,b)\neq (0, 0)$. If $a\neq 0$, then choose $\phi=(\gamma,G)\in \widetilde{\Aut}(\T)$ such that
\[\gamma=a^{\frac{1}{1-p}},\ G=\begin{pmatrix} a^{\frac{1}{p-1}} & -ba^{-1} \\  0 &1  \end{pmatrix}.\]
By \eqref{E:ActionG2}, we have $\phi(P)=(1,0,1,0,0)$. Otherwise $a=0, b\neq 0$, then choose $\phi=(\gamma,G)\in \widetilde{\Aut}(\T)$ such that
\[\gamma=(-b)^{-\frac{1}{p-1}},\ G=\begin{pmatrix} 0 & 1\\ -(-b)^{\frac{1}{p-1}}& 0 \end{pmatrix}.\]
Again, we have $\phi(P)=(1,0,1,0,0)$.

Let $P=(a,b,0,1,0)$. If $(a,b)=(0,0)$, then $P=(0,0,0,1,0)$. Now suppose $(a,b)\neq (0,0)$. If $b\neq 0$, then choose $\phi=(\gamma,G)\in \widetilde{\Aut}(\T)$ such that 
\[
\gamma=1,\ G=\begin{pmatrix} 1&0\\ -b^{-1}a& b^{-1}  \end{pmatrix}.
\]

By \eqref{E:ActionG2}, we have $\phi(P)=(0,1,0,1,0)$. Otherwise $b=0,a\neq 0$, then choose $\phi=(\gamma,G)\in\widetilde{\Aut}(\T)$ such that
\[
\gamma=a^{-\frac{p}{p^2-1}},\ G=\begin{pmatrix} a^{\frac{1}{p^2-1}}& 0 \\ 0 &1  \end{pmatrix}.
\]
Similarly, we have $\phi(P)=(1,0,0,1,0)$.

Let $P=(a,b,1,1,0)$. If $(a,b)=(0,0)$, then $P=(0,0,1,1,0)$. If $b\neq 0$, then choose $\phi\in\widetilde{\Aut}(\T)$ such that
\[
\gamma=b^{-\frac{p}{p^2-p+1}},\ G=\begin{pmatrix} b^{\frac{1}{p^2-p+1}} & 0 \\ -ab^{\frac{1}{p^2-p+1}-1}& b^{\frac{p-1}{p^2-p+1}} \end{pmatrix}.
\]
By \eqref{E:ActionG2}, we have $\phi(P)=(0,1,1,1,0)$. Otherwise $b=0,a\neq 0$, then choose $\phi=(\gamma,G)\in\widetilde{\Aut}(\T)$ such that
\[
 \gamma=a^{-\frac{p}{p^2-1}},\ G=\begin{pmatrix} a^{\frac{1}{p^2-1}} & 0 \\ 0& a^{\frac{1}{p^+1}}  \end{pmatrix}.\]
Hence, we have $\phi(P)=(1,0,1,1,0)$. \\

\bf{Case (\T=\T2):}\rm\ The representative points $P$ have one of the following forms:
\[
(0,b,1,0,0),\ (0,b,0,1,0),\ (0,b,0,0,1),\ (0,b,1,1,0),\ (0,b,1,0,1),\ \mbox{for some}\ b\in \k.
\]
If $P$ is one of the first three cases $(0,b,1,0,0), (0,b,0,1,0), (0,b,0,0,1)$, it is easy to find some $\phi\in \widetilde{\Aut}(\T)$ such that the parameter $b$ can be further taken as $\delta=0,1$. Hence, we obtain the first six isomorphism classes. By Lemma \ref{L:AutNP}, for any $\phi=(\gamma,G)\in \widetilde{\Aut}(\T)$, we can write
\[
G=\begin{pmatrix}
\alpha\gamma & \beta\\
0  & \alpha
\end{pmatrix}
\]
for some $\alpha\neq 0$. Suppose $P=(0,b,1,1,0)$. Then by \eqref{E:ActionG2}, 
\[
\phi(0,b,1,1,0)=\left(0,\, \gamma^p\alpha b,\, (\alpha \gamma)^2,\, \alpha\gamma^{\frac{1+p}{p}},\, 0\right).
\]
Since $(\alpha\gamma)^2=\alpha\gamma^{\frac{1+p}{p}}=1$, we have $\alpha\gamma=\pm 1$. Thus $\gamma^p\alpha=\pm 1$. So the parameter $b$ is parametrized by $\k/\mu_2$. Suppose $P=(0,b,1,0,1)$. Then,
\[
\phi(0,b,1,0,1)=\left(0,\, \gamma^p\alpha b,\, (\alpha \gamma)^2,\, 0,\, \alpha\gamma^{\frac{1}{p}}\right).
\]
By the same reason that $(\alpha\gamma)^2=\alpha\gamma^{\frac{1}{p}}=1$, then $\gamma^p\alpha=1$. So the parameter $b$ is parametrized by $\k$. \\

\bf{Case (\T=\T4):}\rm\ The representative points $P$ have one of the following forms: 
\[
(0,b,1,0,0),\ (0,b,0,0,1),\ (0,b,1,0,1),\ \mbox{for some}\ b\in \k.
\]
Let $P=(0,b,1,0,0)$. If $b=0$, then $P=(0,0,1,0,0)$. If $b\neq 0$, choose $\phi=(\gamma,G)\in \widetilde{\Aut}(\T)$ such that
\[\gamma=b^{-\frac{1}{p-1}},\ G=\begin{pmatrix} 1 & 0\\ 0& b^{\frac{1}{p-1}} \end{pmatrix}.\]
By \eqref{E:ActionG2}, we have $\phi(P)=(0,1,1,0,0)$.

Let $P=(0,b,0,0,1)$. If $b=0$, then $P=(0,0,0,0,1)$. If $b\neq 0$, choose $\phi=(\gamma,G)\in \widetilde{\Aut}(\T)$ such that
\[\gamma=b^{-\frac{p}{p^2-1}},\ G=\begin{pmatrix} 1 & 0\\ 0& b^{\frac{1}{p^2-1}} \end{pmatrix}.\]
Similarly, we have $\phi(P)=(0,1,0,0,1)$.

Let $P=(0,b,1,0,1)$. By Lemma \ref{L:AutNP}, for any $\phi=(\gamma,G)\in \widetilde{\Aut}(\T)$, we have $G=\Diag(\alpha,\beta)$. Thus by \eqref{E:ActionG2},
\[\phi(P)=\left(0,\, \gamma^p\beta b,\, \gamma\alpha\beta,\, 0,\, \gamma^{\frac{1}{p}}\beta\right).\]
Since $\gamma\alpha\beta=\gamma^{\frac{1}{p}}\beta=1$, we have $\beta=\gamma^{-\frac{1}{p}},\gamma=\alpha^{-\frac{p}{p-1}}$. So $\gamma^p\beta b=\alpha^{-1-p}b=\alpha^{-2}b$ for $\alpha\in \mathbb F_p^\times$. Hence the orbits containing $P$ is parametrized by $\k/(\mathbb F_p^\times )^2$. \\

\bf{Case (\T=\T9):}\rm\ The representative points $P$ have one of the following forms:
\[
(a,0,1,0,0),\ (a,0,0,0,1),\ (a,0,1,0,1),\ \mbox{for some}\ b\in \k.
\]
Let $P=(a,0,1,0,0)$. If $a=0$, then $P=(0,0,1,0,0)$. If $a\neq 0$, choose $\phi=(\gamma,G)\in\widetilde{\Aut}(\T)$ such that
\[\gamma=a^{-\frac{p+1}{p^2+p-1}},\ G=\begin{pmatrix} a^{\frac{1}{p^2+p-1}} & 0\\ 0& a^{\frac{p}{p^2+p-1}} \end{pmatrix}.\]
Then by \eqref{E:ActionG2}, we have $\phi(P)=(1,0,1,0,0)$.

Let $P=(a,0,0,0,1)$. If $a=0$, then $P=(0,0,0,0,1)$. If $a\neq 0$, choose $\phi=(\gamma,G)\in\widetilde{\Aut}(\T)$ such that
\[\gamma=a^{-\frac{p^2}{p^3-1}},\ G=\begin{pmatrix} a^{\frac{1}{p^3-1}} & 0\\ 0& a^{\frac{p}{p^3-1}} \end{pmatrix}.\]
Similarly, we have $\phi(P)=(1,0,0,0,1)$.

Let $P=(a,0,1,0,1)$. For any $\phi=(\gamma,G)\in\widetilde{\Aut}(\T)$, we can write $G=\Diag(\alpha,\alpha^p)$. Thus
\[\phi(P)=\left(\gamma^p\alpha a,\, 0,\, \gamma\alpha^{p+1},\, 0,\, \gamma^{\frac{1}{p}}\alpha^p\right).\]
Since $\gamma\alpha^{p+1}=\gamma^{\frac{1}{p}}\alpha^p=1$, so $\gamma=\alpha^{-p-1}$ and $\alpha^{p^2-p-1}=1$. Hence $\gamma^p\alpha a=\alpha^{-p^2-p+1}a=\alpha^{-2p}a$. Moreover, since $2p$ and $p^2-p-1$ are coprime, the orbits containing $P$ is parametrized by $\k/\mu_{(p^2-p-1)}$.

%%%%%%%%%%%%%%%%%%%%%%%%%%%%%%%%

\end{document}